\documentclass[11pt]{amsart}
\usepackage{latexsym,amsmath,amsfonts,amssymb,cite,verbatim,upref}

\title[Constructing Weyl group multiple Dirichlet series]{
Constructing Weyl group multiple Dirichlet series}

\author{Gautam Chinta and Paul E.~Gunnells}

\address{Department of Mathematics,
The City College of CUNY,
New York, NY 10031, USA
}

\email{chinta@sci.ccny.cuny.edu}

\address{Department of Mathematics and Statistics,
University of Massachusetts,
Amherst, MA 01003, USA
}
\email{gunnells@math.umass.edu}

\date{February 24, 2009}
\thanks{Both authors thank the NSF for support.}

\DeclareMathOperator{\ord}{ord}

\DeclareMathOperator{\sgn}{sgn}

\DeclareMathOperator{\Supp}{Supp}

\DeclareMathOperator{\length}{length}

\newcommand{\fin}{\text{fin}}

\newcommand{\re}{\mbox{Re}}

\begin{document}

\begin{abstract}
Let $\Phi$ be a reduced root system of rank $r$.  A {\it Weyl group
multiple Dirichlet series} for $\Phi$ is a Dirichlet series in $r$
complex variables $s_1,\dots,s_r$, initially converging for $\re(s_i)$
sufficiently large, that has meromorphic continuation to ${\mathbb
C}^r$ and satisfies functional equations under the transformations of
${\mathbb C}^r$ corresponding to the Weyl group of $\Phi$.  A
heuristic definition of such series was given in \cite{wmd1}, and they
have been investigated in certain special cases in \cite{wmd1, wmd2,
wmd3, wmd4, wmd5, qmds, cfg, nmdsA2, chbq}.  In this paper we
generalize results in \cite{nmdsA2} to construct Weyl group multiple
Dirichlet series by a uniform method, and show in all cases that they
have the expected properties.
\end{abstract}

\maketitle

\newcommand{\Q}{\ensuremath{{\mathbb Q}}}
\newcommand{\F}{\ensuremath{{\mathbb F}}}
\newcommand{\R}{\ensuremath{{\mathbb R}}}
\newcommand{\A}{\ensuremath{{\mathbb A}}}
\newcommand{\C}{\ensuremath{{\mathbb C}}}
\newcommand{\Z}{\ensuremath{{\mathbb Z}}}
\newcommand{\x}{\ensuremath{{\bf { x}}}}
\newcommand{\y}{\ensuremath{{\bf { y}}}}
\newcommand{\s}{\ensuremath{{\bf { s}}}}
\newcommand{\bd}{\overline{d}}
\newcommand{\p}{\ensuremath{{\varpi}}}
\newcommand{\m}{\ensuremath{{\mathfrak { m}}}}
\newcommand{\g}{\ensuremath{{\mathfrak { g}}}}
\newcommand{\bfm}{\ensuremath{{{\bf m}}}}
\newcommand{\bfa}{\ensuremath{{{\bf a}}}}
\newcommand{\fb}{\ensuremath{{\mathfrak { b}}}}
\newcommand{\f}{\ensuremath{{\mathfrak { f}}}}

\newcommand{\fS}{\ensuremath{{\mathfrak { S}}}}
\newcommand{\PP}{\ensuremath{{\mathcal P}}}
\newcommand{\QQ}{\ensuremath{{\mathcal Q}}}
\newcommand{\ZZ}{\ensuremath{{\mathcal Z}}}
\newcommand{\TT}{\ensuremath{{\mathcal T}}}
\newcommand{\D}{\ensuremath{{\mathcal D}}}
\newcommand{\M}{\ensuremath{{\mathcal M}}}
\newcommand{\MM}{\ensuremath{{\mathfrak M}}}

\newcommand{\E}{\ensuremath{{\mathcal E}}}
\newcommand{\psia}{\ensuremath{{\it \Psi}}}
\newcommand{\innprod}[2]{\langle #1, #2 \rangle}

\newcommand{\im}{\mbox{Im}}
\newcommand{\res}{\mathop{\mbox{Res}}}
\newcommand{\half}{{\smfrac{1}{2}}}

\newcommand{\I}{{\mathcal I}}
\newcommand{\J}{{\mathcal J}}
\newcommand{\N}{{\mathcal N}}
\newcommand{\pr}{{\prime}}
\newcommand{\norm}[1]{\left\Vert #1\right\Vert}
\newcommand{\nrm}[1]{\left| #1\right|}
\newcommand{\abs}[1]{\left|#1\right|}

\newcommand{\prs}[2]{\Bigl(\frac{#1}{#2}\Bigr)}
\newcommand{\smfrac}[2]{{\textstyle \frac{#1}{#2}}} 
\newcommand{\nn}{\nonumber}

\newcommand{\OO}{{\mathcal O}}
\newcommand{\OOmon}{{\OO_{\text{mon}}}}
\newcommand{\hpl}{h^{(\p,l)}}
\newcommand{\fpl}{f^{(\p,l)}}

\newcommand{\polyring}{A}
\newcommand{\local}{\widetilde{A}}
\newcommand{\wt}{\omega }

\renewcommand{\mod}{\bmod}

\newtheorem{theorem}{Theorem}
\newtheorem{lemma}[theorem]{Lemma}
\newtheorem{prp}[theorem]{Proposition}
\newtheorem{corollary}[theorem]{Corollary}
\newtheorem{conjecture}[theorem]{Conjecture}
\newtheorem{fact}[theorem]{Fact}
\newtheorem{claim}[theorem]{Claim}
\theoremstyle{definition}

\newtheorem{example}[theorem]{Example}
\newtheorem{remark}[theorem]{Remark}
\newtheorem{defin}[theorem]{Definition}

\numberwithin{theorem}{section}
\numberwithin{equation}{section}

\section{Introduction}

This paper describes a technique to construct Weyl group multiple
Dirichlet series.  Such series were first introduced in \cite{wmd1},
which also described a heuristic means to define, analytically
continue, and prove functional equations for a family of Dirichlet
series in several complex variables.  Several subsequent papers have
dealt with the problem of making the heuristic definitions precise and
completing the proofs of analytic continuation and functional equation
of these Weyl group multiple Dirichlet series along the lines
suggested in \cite{wmd1}.  Before listing some of the partial results
obtained in these papers, we say a bit more about the type of multiple
Dirichlet series studied.

Let $F$ be an algebraic number field containing the $2n^{th}$ roots of
unity.  Fix a finite set of places $S$ containing all the archimedean
places and those that are ramified over $\Q.$ Take $S$ large enough
that $\OO_S$, the ring of $S$-integers of $F$, has class number one.

Let $\Phi$ be a reduced root system of rank $r.$ Let ${\bf m}=(m_1,
\ldots, m_r)$ be an $r$-tuple of integers in $\OO_S$ and $\s=(s_1,
\ldots , s_n)$ an $r$-tuple of complex variables.  In \eqref{eq:18} we
define a certain finite-dimensional vector space $\mathcal {M}
(\Omega, \Phi)$ of complex-valued functions on $(F_{S}^{\times})^{r}$,
and we choose $\psia \in \mathcal {M} (\Omega, \Phi)$.  In Section
\ref{section:new} we define coefficients $H (\mathbf{c}; \bfm)$, where
$\mathbf{c}$ and $\bfm$ are $r$-tuples of nonzero integers in
$\OO_{S}$.  To this data we associate a multiple Dirichlet series in
$r$ complex variables
\begin{equation}
  \label{eq:2}
  Z(\s;{\bf m},\psia)=Z(\s; {\bf m}, \psia;
\Phi, n)=\sum_{{\bf c}} \frac{H({\bf c};{\bf m})
    \psia({\bf c})} {\prod |c_i|^{s_i}} ,
\end{equation}
where $\mathbf{c} = (c_{1},\dotsc ,c_{r})$ and each $c_{i}$ ranges
over $(\OO_S\smallsetminus \{0 \})/\OO_S^{\times}$.  (Although the
coefficients $H$ and $\psia$ are defined for nonzero $c_{i}\in
\OO_{S}$, their product $H\psia$ is unchanged when any $c_{i}$ is
multiplied by a unit, as the arguments in \cite{wmd2} show.)  The sum
\eqref{eq:2} is absolutely convergent when $\re (s_i)> 3/2$.  We call
\eqref{eq:2} a \emph{degree $n$ Weyl group multiple Dirichlet series
of type $\Phi$ with twisting parameter ${\bf m}.$}

These series were introduced in the papers \cite{wmd1, wmd3, wmd4}.
Actually, except for the cases listed below, in these papers the
coefficients
$$H(c_1, \ldots, c_r;m_1,\ldots, m_r)$$ were defined only when the
product $c_1\cdots c_r\cdot m_1\cdots m_r$ was squarefree---we call
these the {\em powerfree} coefficients---and it was suggested that it
should be possible to define the $H$-coefficients for all $r$-tuples
$(c_1,\ldots, c_r)$ in such a way that the resulting multiple
Dirichlet series has an analytic continuation to $\C^r$ and satisfies
a group of functional equations isomorphic to the Weyl group $W$ of
$\Phi.$  These powerfree coefficients can be expressed in terms of
$n$th order Gauss sums and the root data.

The full definitions and expected properties of $Z(\s; {\bf m}, \psia;
\Phi, n)$ have been given in the following cases:
\begin{enumerate} 
\item \cite{wmd1} $\Phi=A_2$, $n$ arbitrary, ${\bf
m}=(1,1)$. 
\item \cite{wmd2, wmd4} $\Phi$ arbitrary and the degree $n$
  sufficiently large with respect to $\Phi$ and the twisting parameter
  {\bf m}.  In \cite{wmd2, wmd4}, this is called the {\em
    stable} case.
\item \cite{wmd3} $\Phi=A_2$, $n$ arbitrary, ${\bf m}$ arbitrary.
\item \cite{qmds, cfg} $\Phi$ simply-laced, $n=2$, i.e.~the quadratic
case.  (The methods of\cite{qmds}, which dealt only with the untwisted case,
were extended to arbitrary ${\bf m}$ in \cite{cfg}.)
\item \cite{wmd5, wmdbook}  $\Phi=A_r$, $n$ arbitrary, ${\bf
m}$ arbitrary.  In these preprints, Brubaker, Bump and Friedberg deal
with the type $A_r$ multiple Dirichlet series in two different
ways. In \cite{wmd5} they show that the multiple Dirichlet series
coincide with Whittaker coefficients of a metaplectic Eisenstein
series, while in \cite{wmdbook}, they use a combinatorial approach
based on crystal graphs.  At present, both methods work only for the
type $A$ root systems.
\end{enumerate}

In this paper, we give by a uniform method, a complete definition of
the Weyl group multiple Dirichlet series and proofs of the expected
properties in all cases.  That is, we have the following theorem
(cf.~Theorem \ref{thm:maintheorem}):

\begin{theorem}
  Let $n\geq 1$ and let $F, S, \OO_S$ be as above.  Let $\Phi$ be a
  reduced root system of rank $r$ and let ${\bf m}$ be an $r$-tuple of
  nonzero integers in $\OO_S.$ Then, the multiple Dirichlet series
  $Z(\s; {\bf m}, \psia;\Phi, n) $ defined in Section \ref{sec:wmds}
  has an analytic continuation to $\s\in \C^r$ and satisfies a group of
  functional equations isomorphic to $W$, the Weyl group of $\Phi.$
\end{theorem}

We now describe our method of proof.  As stated above, the heuristic
definition of $Z(\s;\bfm , \psia )$ includes the definition of the
powerfree coefficients $H(c_1, \ldots, c_r;m_1, \ldots
m_r)$. Moreover, the coefficients satisfy a twisted multiplicativity
\eqref{eq:3}--\eqref{eq:5} that reduces their specification to that of
the prime power coefficients $H(\p^{k_1}, \ldots, \p^{k_r};\p^{l_1},
\ldots, \p^{l_r})$, for $\p$ prime in $\OO_S.$ We are naturally led to
consider the generating series
\begin{equation}
  \label{eq:16}
N=  N(x_1, \ldots, x_r)=\sum_{k_1, \ldots, k_r\geq 0} H(\p^{k_1},
  \ldots, \p^{k_r})x_1^{k_1}\cdots x_r^{k_r}, 
\end{equation}
where we have suppressed the twisting parameter from the notation.
This generating series, which turns out to be a polynomial in the
indeterminates $x_1,\ldots, x_r$, must be defined in such a way that
the resulting global object $Z(\s;\bfm , \psia )$ can be shown to
satisfy the right functional equations.  

The main idea of this paper is that there is a certain action of $W$
on the field of rational functions $\C (x_{1},\dotsc
,x_{r})$.  Using this action we can construct an invariant function
$h$ with an easily understood denominator $D \in \C [x_{1},\dotsc
,x_{r}]$.  We then take $N$ to be the polynomial $hD$, and prove that
the multiple Dirichlet series constructed using \eqref{eq:16}
satisfies the requisite functional equations under $W$. What makes
things difficult is that the $W$-action is a little
complicated.  Moreover, it is not obvious that this action has
anything to do with the functional equations of a global multiple
Dirichlet series.

There are at least three ways to motivate the $W$-action
defined in Section \ref{sec:Waction}.  First, if the functional
equations of $Z(\s;\bfm , \psia)$ are assumed, and one works backwards
through the proof of Theorem \ref{thm:FEa}, one finds that the
$\p$-parts of the multiple Dirichlet series must have precisely the
$W$-invariance defined in Section \ref{sec:Waction}.  For $n=2$ and
$\Phi=A_2$, this is the approach taken in Section 5 of
\cite{cfhSurvey}.  A second way to motivate the action of $W$ is the
remarkable fact that an untwisted Weyl group multiple Dirichlet series
over the rational function field coincides (after a simple change of
variables) with its own $\p$-part.  This was noted in \cite{cfhSurvey}
for $\Phi=A_2$ and $n=2$ and in \cite{ffwmd} for $\Phi=A_2$ and
general $n$.  Consequently, in this context, the $\p$-parts satisfy
essentially the same functional equations as the multiple Dirichlet
series itself.  These functional equations follow from the functional
equations of Kubota's Dirichlet series.  Over the rational function
field, Kubota's theory has been very explicitly worked out by
Hoffstein \cite{jhoff} and Patterson \cite{sjpNote}.  This is the
point of view which motivated our work in \cite{nmdsA2}.  Yet another
another approach to motivating the local functional equations is via
the {\em Eisenstein conjecture} described in \cite{wmd2}.  This
conjecture essentially states that a Weyl group multiple Dirichlet
series should be a Whittaker coefficient of an Eisenstein series on an
$n$-fold metaplectic cover of $G$, the simply-connected algebraic
group over $F$ whose root system is the dual of $\Phi.$ If this is
true, then the $\p$-part of a multiple Dirichlet series should be
related to a local Whittaker function of a $\p$-adic metaplectic
group.  The functional equations satisfied by a $\p$-adic metaplectic
Whittaker function have been described by Kazhdan--Patterson
\cite{kp}.  At least for root systems of type $A$, the group action we
define in Section \ref{sec:Waction} coincides with the functional
equations in Lemma 1.3.3 of \cite{kp}. From this perspective, our
averaging formula (\ref{eqn:defofh}) can potentially be seen as a
Casselman--Shalika formula for $\p$-adic Whittaker coefficients on
metaplectic groups.  This connection is being further developed in
current work of Chinta and Offen.

We conclude this introduction with a description of the organization
of the paper.  In Section \ref{sec:preliminaries} we review basic
facts about Gauss sums and reciprocity that we will need.  Section
\ref{sec:Waction} is devoted to the definition of the action of the
Weyl group on the field of rational functions $\C(x_1, \ldots, x_r)$.
We further construct rational functions invariant under the group
action.  These invariant functions play a key role in the definition
of the Weyl group multiple Dirichlet series.  The results of this
section depend only on the combinatorics of the root system and its
Weyl group.

Section \ref{section:new} describes how to define the coefficients
$H({\bf c};{\bf m})$ of the multiple Dirichlet series.  In short, the
prime power coefficients 
\[
H(\p^{\beta_1},
\ldots, \p^{\beta_r}; \p^{l_1}, \ldots ,\p^{l_r}) 
\]
can be read off
from the Taylor coefficients of the invariant functions constructed in
Section \ref{sec:Waction}.  The general coefficient $H({\bf c};{\bf
  m})$ is then defined in terms of the prime power coefficients via the
``twisted multiplicity'' given in (\ref{eq:3}).  In addition we show
that certain combinations $f^{\p;\bf k}$ of the local factors satisfy
the simple functional equation given in Theorem \ref{thm:f}.

Section \ref{sec:kubota} begins with a review of Kubota's Gauss sum
Dirichlet series, as presented in \cite{bb}.  These Dirichlet series,
constructed with Gauss sum coefficients, have a meromorphic
continuation to the complex plane and satisfy a functional equation as
$s\mapsto 2-s,$ see Proposition \ref{prp:KubotaFE}.  We further
construct a single-variable Dirichlet series $\E$ (defined in
(\ref{eq:8}) using the coefficients $H({\bf c};{\bf m})$, and show that
it can be written as a linear combination of products of Kubota's
series and the local factors  $f^{\p;\bf k}$ defined in Section
\ref{section:new}.  Consequently, $\E$ satisfies the functional
equation described in Theorem \ref{thm:FEa}.

In Section \ref{sec:wmds} we define the multiple Dirichlet series $Z(\s; {\bf
  m}, \Psi)$ and prove our main result, Theorem
\ref{thm:maintheorem}, which gives the functional equation and
meromorphic continuation of $Z.$  The single-variable Dirichlet series
$\E$ act as the building blocks of $Z$, and hence the functional
equations of the multivariable object follow from those of $\E$
established in the previous section.  Then the meromorphic
continuation of the multiple Dirichlet series $Z$ follows from 
 a simple convexity argument, as has been extensively used in the
 recent papers \cite{qmds,wmd3,wmd1,wmd2,nmdsA2}.

In Section \ref{sec:examples} we suggest how other examples of
multiple Dirichlet series that have previously appeared in the
literature fall into the framework presented here, and make some
further remarks.  This final section
is not meant to be absolutely precise or definitive, but rather
indicative of further prospects for research in the field.

As the outline given above suggests, we have tried to minimize the
interdependence of the sections on one another.  In particular the
only result from Section \ref{section:new} that is used in the later
sections is Theorem \ref{thm:f}, which appears at the very end of the
proof of Theorem \ref{thm:FEa}.  Theorem \ref{thm:FEa} is then the
only theorem from the earlier sections used in the proof of our main
result Theorem \ref{thm:maintheorem}.  The structure of this argument
mimics the pattern established in our earlier papers
\cite{qmds,chbq,nmdsA2}, where functional equations of local factors
induce functional equations in certain univariate Dirichlet series,
which in turn induce functional equations in the multivariate Weyl
group multiple Dirichlet series that are our main objects of
interest.  This approach has by now become streamlined and dense, as
necessitated by the increasing notational complexities.  The reader
who is new to these techniques may benefit from the presentation in
\cite{nmdsA2}, where the case of the root system $A_2$ is worked out
in detail.

{\bf Acknowledgments.}  The authors are deeply grateful to Ben 
Brubaker, Dan Bump, Sol Friedberg, Jeff Hoffstein, Joel Mohler,
and Samuel Patterson for their advice and extensive correspondence
throughout the preparation of this work.  The authors also thank the
referee for several comments that greatly improved our paper.

\section{Preliminaries}\label{sec:preliminaries}
In this section we recall basic facts about Hilbert symbols and Gauss
sums we will need in the sequel.  Our exposition follows
\cite[\S\S2.3--2.7]{wmd4} essentially verbatim.

Recall that $F$ is a number field containing the $2n$-th roots of
unity, and that $S$ is a finite set of places containing all the
archimedean places and those ramified over $\Q$.  Recall also that $S$
is assumed to be large enough that $\OO_S$, the ring of $S$-integers
of $F$, is a principal ideal domain.  Let $F_{S} = \prod_{v\in S}
F_{v}$.  Similarly let $S_{\fin}\subset S$ be the subset of finite
places, and let $F_{\fin} = \prod_{v\in S_{\fin }} F_{v}$.  We embed
$\OO_{S}\rightarrow F_{S}$ diagonally.

Each place $v$ determines a local Hilbert symbol
$(\phantom{a},\phantom{a})_{v}\colon F_{v}^{\times}\times
F_{v}^{\times}\rightarrow \mu_{n}$, where $\mu_{n}$ is the group of
$n$-th roots of unity \cite{neukirch}.  We assume fixed an embedding
$\epsilon$ from $\mu_n$ to $\C^\times$ and identify $\mu_n$ with its
image.   The Hilbert symbol leads to a pairing
\[
(\phantom{a},\phantom{a})_{S}\colon F_{S}^{\times}\times
F_{S}^{\times}\rightarrow \mu_{n},
\]
defined by $(a,b)_{S} = \prod_{v\in S} (a,b)_{v}$. 

A subgroup $\Omega\subset F_{S}^{\times }$ is called \emph{isotropic}
if $(\phantom{a},\phantom{a})_{S}|_{\Omega \times \Omega}$ is
trivial.  Now let $\Omega$ be the subgroup
$\OO_{S}^{\times}F_{S}^{\times n}$, which is maximal isotropic.
Given a positive integer $t$, define the complex vector space $\M_{t}
(\Omega)$ by
\[
\M_{t} (\Omega) = \Bigl\{\Psi \colon F_{\fin}^{\times}\rightarrow \C
\Bigm| \Psi (\varepsilon c) = (c,\varepsilon)_{S}^{-t} \Psi
(c)\text{\ for all $\varepsilon \in \Omega $}\Bigr\}.
\]
We abbreviate $\M_{1} (\Omega)$ by $\M (\Omega)$. 
Note that if
$\varepsilon$ is sufficiently close to the identity in
$F_{\fin}^{\times}$, then $\varepsilon$ is an $n$-th power at every
finite place $S_{\fin}$.  Hence functions in $\M (\Omega )$ are locally
constant. One can show that the dimension of $\M (\Omega) =
[F_{S}:\Omega]$, which is finite.

Let $a\in \OO_{S}$ and let $\fb \subset \OO_{S}$ be an ideal.  Let
$(\smfrac{a}{\fb})$ be the $n$-th power residue symbol defined in
\cite{wmd2}.  In general this symbol depends on the set of places $S$,
but we omit this from the notation.  Let $t$ be a positive integer and
let $a,c\in \OO_{S}$ with $c\not =0$.  Choose a nontrivial additive
character $\psi$ of $F_{S}$ such that $\psi (x\OO_{S})=1$ if and only
if $x\in \OO_{S}$ \cite[Lemma 1]{bb}.  Then we define a Gauss sum $g
(a,c;\epsilon^t)$ by
\begin{equation}
  \label{eqn:gausssum}
g (a,c;\epsilon^t) = \sum_{d\mod c} \epsilon^t\left(
\prs{d}{c\OO_{S}}\right) \psi \prs{ad}{c}.  
\end{equation}


We list some properties of the residue symbol and Gauss sums.  For
proofs,  see \cite[Chapter 6/Theorem 8.3]{neukirch}
for reciprocity, and Ireland--Rosen \cite{ir} for the properties of
Gauss sums.  

\begin{prp} Let $a,b,m \in \OO_S$ with $a,b$ relatively prime. Then
  \begin{enumerate}
\item {} [Reciprocity] 
  $\prs{a}{b}=(b,a)_S\prs {b}{a}.$
\item $g(m,ab;\epsilon^t)=g(m,a;\epsilon^t)g(m,b;\epsilon^t)\epsilon^t\left(\prs ab \prs ba\right)$ 
\item $g(am,b;\epsilon^t)=\epsilon^{-t}\left(\prs ab\right) 
g(m,b;\epsilon^t) $
\item If $\p\in \OO_S$ is prime and $t$ does not vanish mod $n$, then 
\[
g (1,\p;
\epsilon^{t})g (1,\p; \epsilon^{-t}) = |\p|.
\]
\end{enumerate}
\end{prp}

\section{A Weyl group action}\label{sec:Waction}
We begin by recalling some basic definitions and properties of root
systems.  For more details we refer to \cite{humph}.

Let $\Phi$ be an irreducible reduced root system of rank $r$ with Weyl
group $W$. Choose an ordering of the roots and let $\Phi = \Phi^+ \cup
\Phi^{-}$ be the decomposition into positive and negative roots.  Let
$\{\alpha_1, \alpha_2, \ldots, \alpha_r\}$ be the set of
simple roots, and let $\sigma_i$ be the Weyl group element
corresponding to the reflection through the hyperplane perpendicular
to $\alpha_i$.   Let $\{\wt_{1},\dotsc ,\wt_{r} \}$ be the
fundamental weights.

Let $\Lambda$ be the lattice generated by the roots.  The Weyl group
$W$ acts on $\Lambda$ from the left: $\lambda \mapsto w\lambda$.  We
choose a $W$-invariant inner product
$\innprod{\phantom{a}}{\phantom{a}}$ on $\Lambda\otimes \R$ normalized
so that the short roots of $\Phi$ have squared length $1$, that is,
$\innprod{\alpha}{\alpha}=1$ for all short roots $\alpha $.  With this
convention, the inner product of any two roots is always a
half-integer.\footnote{We remark that this is different from the usual
Bourbaki convention, which takes the normalization
$\innprod{\alpha}{\alpha} = 2$ for all short roots \emph{except} those
in type $B_r$, where the short roots satisfy $\innprod{\alpha}{\alpha
}=1$.}  Our normalization implies that for $\alpha \in \Phi$, we have
\[
\norm{\alpha}^{2} = \begin{cases}
1&\text{for all $\alpha$ in types $A, D, E$,}\\
1&\text{for $\alpha$ a short root in types $B, C, F_{4}, G_{2}$,}\\
2&\text{for $\alpha$ a long root in types $B, C, F_{4}$,}\\
3&\text{for $\alpha$ a long root in type $G_{2}$}.
\end{cases}
\]

Label the nodes of the Dynkin diagram of $\Phi$.  We say that nodes
$i$ and $j$ are \emph{adjacent} if $i\neq j$ and $(\sigma_i
\sigma_j)^2\not =1$.  The group $W$ is generated by the simple
reflections $\sigma_{i}$, which satisfy the relations
$(\sigma_i\sigma_j)^{r(i,j)} = 1$, where for $1\leq i, j\leq r$ we
have
\begin{equation}\label{eqn:WRelations}
r(i,j)=\left\{\begin{array}{ll}
1 & \text{if $i=j$,}\\
2 & \text{if $i, j$ are not adjacent,}\\
3 & \text{if $i, j$ are adjacent and $\alpha_{i}$, $\alpha_{j}$
have the same length,}\\
4 & \text{if $i, j$ adjacent, $\Phi \not = G_{2}$ and $\alpha_{i}$, $\alpha_{j}$
have different lengths,}\\
6 & \text{for $i, j$ adjacent in type $G_{2}$.}
\end{array}\right.
\end{equation}

Let $c (i,j)= 2
\innprod{\alpha_{i}}{\alpha_{j}}/\innprod{\alpha_{j}}{\alpha_{j}}\in
\Z$ be the Cartan integer attached to the simple roots $\alpha_{i}$,
$\alpha_{j}$.  These integers encode the action of $W$ on the simple
roots:
\begin{equation}\label{eqn:WactionRoots}
\sigma_{j}\colon \alpha_{i} \longmapsto \alpha_{i} - c (i,j)\alpha_{j}.
\end{equation}

Let $\length\colon W\rightarrow \Z_{\geq 0}$ be the length function with
respect to the generators $\sigma_{1},\dotsc ,\sigma_{r}$, and put
$$\sgn(w)=(-1)^{\length(w)}.$$
For any $w\in W$, define $\Phi (w)$ by
\[
\Phi (w) = \{\alpha \in \Phi^{+} \mid w\alpha \in \Phi^{-} \}.
\]
We have $\length (w) = |\Phi (w)|$.
If $w\in W$ satisfies $\length (\sigma_{i}w)=\length (w) + 1$ for a
simple reflection $\sigma_{i}$, then 
\begin{equation}\label{eqn:phiwleft}
\Phi (\sigma_{i}w) = \Phi (w) \cup \{w^{-1} \alpha_{i}\}.
\end{equation}
Similarly, if $\length (w\sigma_{i})=\length (w) + 1$ for a
simple reflection $\sigma_{i}$, then
\begin{equation}\label{eqn:phiwright}
\Phi(w\sigma_i)=\sigma_i\left(\Phi(w)\right)\cup \{\alpha_i\}.
\end{equation}

Moreover, define $\rho \in \Lambda \otimes \Q$ by $\rho =
\sum \wt_{i} $.  Then
\[
\rho -w^{-1}\rho  = \sum_{\alpha \in \Phi (w)} \alpha.
\]

Any $\lambda \in \Lambda$ has a unique representation as an integral
linear combination of the simple roots
\begin{equation}\label{eqn:repn}
\lambda =k_1\alpha_1+k_2\alpha_2+\cdots+k_r\alpha_r.
\end{equation}
Let
$$d(\lambda )=k_1+\dotsb +k_{r}$$ be the usual height function on
$\Lambda$.  Introduce the standard partial ordering on $\Lambda$ by
defining $\lambda\succeq 0$ if $\lambda$ is a nonnegative linear
combination of the simple roots.  Given $\lambda , \lambda '\in
\Lambda$, define $\lambda \succeq\lambda '$ if $\lambda -\lambda'
\succeq 0$.

Our goal in this section is to study a certain Weyl group action.  In
fact we will define a collection of Weyl group actions indexed by some
parameters.  The first parameter is a strongly dominant weight $\theta
$.  The weight $\theta$ determines an $r$-tuple $\ell=(l_1, \ldots ,
l_r) $ of nonnegative integers by $\theta = \sum
(l_{i}+1)\wt_{i}$.  We call $\ell$ a \emph{twisting
parameter}; ultimately it will be connected to the global twisting
parameter $\bfm$ in the construction of our multiple Dirichlet series.
The weight $\theta$ determines an action of $W$ on the root
lattice through affine linear transformations by 
\begin{equation}\label{eqn:dotaction}
w\bullet\lambda  = w (\lambda -\theta) + \theta.
\end{equation}
It is not hard to check that for $w=\sigma_{i}$, a simple reflection,
we have
\[
\sigma_{i}\bullet \lambda  = \sigma_{i}\lambda + (l_{i}+1)\alpha_{i}.
\]

The next parameter is a positive integer $n$.  The integer $n$
determines a collection of integers $\{m (\alpha)\}_{\alpha \in
\Phi^{+}}$ by 
\begin{equation}\label{eqn:defofm}
m(\alpha)=n/\gcd(n, \norm{\alpha}^{2}).
\end{equation}

Next we choose a positive integer $q$.  Together with $q$ we consider
a collection of complex numbers $\gamma (i)\in \C$, indexed by the
integers modulo $n$, and such that $\gamma (0) = -1$ and
\[
\gamma (i)\gamma (-i) = 1/q\quad \text{if $i\not = 0 \mod n$.}
\]
Later $q$ will be taken to be the norm of a prime $\p$ in
$\OO_{S}$, and up to a factor of $q$ the number $\gamma (i)$ will be the
Gauss sum $g (1,\p ; \epsilon^{i})$ from \eqref{eqn:gausssum}.

We are almost ready to define our action.  Choose and fix parameters
$(\theta , n, q, \{\gamma (i) \})$ as above.  Let $\polyring = \C
[\Lambda]$ be the ring of Laurent polynomials on the lattice
$\Lambda$.  Hence $\polyring$ consists of all expressions of the form
$f = \sum_{\beta \in \Lambda} c_{\beta} \x^{\beta}$, where $c_{\beta}\in
\C$ and almost all are zero, and the multiplication of monomials is
defined by addition in $\Lambda$: $\x ^{\beta}\x^{\lambda} = \x^{\beta
+\lambda}$.  Given $f$, the set of $\{\beta \mid c_{\beta}\not =0 \}$
is called the support of $f$, and is denoted $\Supp f$.  We 
identify $\polyring$ with $\C [x_{1}, x_{1}^{-1}, \dotsc , x_{r},
x_{r}^{-1}]$ via $\x^{\alpha_{i}} \mapsto x_{i}$.

We define a ``change of variables'' action on $\polyring$ as follows.  
Write $\x=(x_1,\ldots,
x_r)$.  Then we define $\sigma_j\x=\x^\prime$, where
\begin{equation}\label{eqn:wiaction1}
(\x^\prime)_{i}=q^{-c (i,j)}x_{i}x_{j}^{-c (i,j)}.
\end{equation}
Note that $(\sigma_{j}\x)_{j} = 1/ (q^{2}x_{j})$ for all $j$, and
$(\sigma_{j}\x )_{i} =x_{i}$ if and only if $\alpha_{i}$ and
$\alpha_{j}$ are orthogonal in $\Phi$.  It is easy to verify that this
action of the simple reflections extends to all of $W$, since it is
essentially a reformulation of the standard geometric action of $W$ on $\Lambda
\otimes \R$ (cf.~\eqref{eqn:WactionRoots}).  One can also easily check 
that if $f_{\beta }
(\x) = \x^{\beta}$ is a monomial, then
\begin{equation}\label{eqn:monomialaction}
f_{\beta} (w\x) = q^{d (w^{-1}\beta -\beta)}\x^{w^{-1}\beta}.
\end{equation}

Now let $\Lambda '\subset \Lambda$ be the sublattice generated by the
set $\{m (\alpha) \alpha \}_{\alpha \in \Phi}$.  A direct computation
with Cartan matrices shows that $W$ takes $\Lambda '$ into itself.
Let $\local$ be the field of fractions of $A$.  We have the
decomposition
\begin{equation}\label{eqn:decomposition}
\local = \bigoplus_{\lambda \in \Lambda /\Lambda '} \local_{\lambda },
\end{equation}
where $\local_{\lambda }$ consists of the functions $f/g$ ($f,g\in \polyring $)
such that $\Supp g$ lies in the kernel of the map $\nu \colon \Lambda
\rightarrow \Lambda /\Lambda '$, and $\nu$ maps $\Supp f$ to $\lambda
$.

We now define the action of $W$ on $ \local $ for a generator
$\sigma_i\in W$.  Put $m = m (\alpha_{i})$ (cf.~\eqref{eqn:defofm}).
For any $\beta \in \Lambda$, define the integer
\begin{equation}\label{label:eqn:defofmu}
  \mu _{\ell ,i} (\beta) = d (\sigma_{i}\bullet \beta -\beta) \in \Z.
\end{equation}
Let $(k)_m\in \{0,\dotsc ,m-1 \}$ denote the remainder upon division
of $k$ by $m$, and define rational functions
\begin{align}
  \PP_{\beta ,\ell ,i} (x) &= (qx)^{l_{i}+1- (\mu _{\ell ,i} (\beta))_{m}}\frac{1-1/q}{1- (qx)^{m}/q},\label{eqn:Peqn}\\
  \QQ_{\beta , \ell , i} (x) &= -\gamma(-\norm{\alpha_{i}}^{2}
\mu_{\ell ,i}(\beta)) (qx)^{l_{i}+1-m}\frac{1- (qx)^{m}}{1- (qx)^{m}/q}.\label{eqn:Qeqn}
\end{align}

\begin{defin}\label{def:wiaction}
Let $f (\x) \in \local_{\beta }$.  Define
\begin{equation}\label{wiaction}
(f|_\ell \sigma_i) (\x ) = (\PP_{\beta,\ell ,i} (x_{i}) +
\QQ_{\sigma_{i}\bullet \beta ,\ell ,i} (x_{i}))f (\sigma_{i}\x )\in \local .
\end{equation}
We extend this definition linearly to all of $\local$ using
\eqref{eqn:decomposition}.
\end{defin}

\begin{theorem}\label{thm:waction}
The action of the generators \eqref{wiaction} extends to give an
action of $W$ on $\local $.  More precisely, the action
\eqref{wiaction} satisfies the defining relations
\eqref{eqn:WRelations}.
\end{theorem}

\begin{proof}
The proof consists of explicit computations that are very similar to
those done in the proof of \cite[Lemma~3.2]{qmds}.  The main point is
that verifying the relations \eqref{eqn:WRelations} amounts to
checking certain identities among rational functions.  These
identities depend only in a minor way on $n$ and are easily
implemented on a computer.  Since the computations are rather lengthy,
we content ourselves with only explicitly presenting some of them
here.  For the general computations, we merely give an overview and
will leave most of the details to the reader.

First we show that the $\sigma_{i}$ act by involutions on $\local $.  Let
$m=n/\gcd (n,\norm{\alpha_{i}}^{2})$ as above. 
It suffices to check on functions of the form $f (\x)=\x^{\beta }/h
(\x)$, where $h (\x)\in \local_{0}$.  We compute
\eqref{wiaction} and separate the result into homogeneous terms, and find
\begin{align}\label{eq:first}
\PP_{\beta,\ell ,i} (x_{i}) f (\sigma_{i}\x )&\in \local_{\beta}\\
\QQ_{\sigma_{i}\bullet \beta ,\ell ,i} (x_{i}) f (\sigma_{i}\x)&\in \local_{{\sigma_{i}\bullet\beta}},
\end{align}
where we have abused notation and have denoted the image of $\beta$ in
$\Lambda /\Lambda '$ by $\beta$ as well.  Applying $\sigma_{i}$ again,
we find
\begin{multline}\label{eq:second}
f|_{\ell}\sigma_{i}^{2}  = \Bigl(\PP_{\beta} (x)\PP_{\beta} (1/(q^{2}x)) +
\QQ_{\sigma_{i}\bullet \beta }(x)\PP_{\beta} (1/(q^{2}x)) \\
+ \PP_{\sigma_{i}\bullet \beta } (x)\QQ_{\sigma_{i}\bullet \beta }
(1/(q^{2}x)) + \QQ_{\beta} (x) \QQ_{\sigma_{i}\bullet \beta } (1/(q^{2}x))\Bigr)
f.
\end{multline}
In \eqref{eq:second} we lightened the notation by removing $\ell ,i$
from $\PP , \QQ$ and by writing $x$ for $x_{i}$.  Now using the
identities
\[
\mu_{\ell ,i} (\beta) = -\mu_{\ell ,i} (\sigma_{i}\bullet \beta )
\]
and 
\[
(-k)_{m} = \begin{cases}
0&\text{if $m|k$,}\\ 
m-k&\text{otherwise},
\end{cases}
\]
we can check that the sum of the four products of $\PP, \QQ$ on the right of
\eqref{eq:second} equals $1$.  Indeed, there are two cases to
consider, according to whether $\mu_{\ell ,i} (\beta)$ vanishes mod
$m$ or
not.  In the second case we have 
\begin{align*}
\QQ_{\sigma_i\bullet\beta}(x)\PP_{\beta} (1/(q^{2}x))
+ \PP_{\sigma_i\bullet\beta} (x)\QQ_{\sigma_i\bullet\beta} (1/(q^{2}x)) &= 0,\\
 \PP_{\beta} (x)\PP_{\beta} (1/(q^{2}x)) + \QQ_{\beta} (x) \QQ_{\sigma_i\bullet\beta}
(1/(q^{2}x)) &= 1,
\end{align*}
and in the first case there are nontrivial cancellations among all
four terms that lead to the desired sum.  This shows that the
$\sigma_{i}$ act by involutions.

Now we explain how to prove that the rest of \eqref{eqn:WRelations}
hold.  We consider a relation of the form
$\sigma_{i}\sigma_{j}\sigma_{i} = \sigma_{j}\sigma_{i}\sigma_{j}$.
This is only possible if $\norm{\alpha_{i}}=\norm{\alpha_{j}}$, and so
without loss of generality we can assume $m=n$.

We again put $f (\x)=\x^{\beta}/h (\x )$ and compute $f_{1} (\x
):=(f|_{\ell}\sigma_{i}\sigma_{j}\sigma_{i} )(\x )$ and $f_{2} (\x )
:= (f|_{\ell}\sigma_{j}\sigma_{i}\sigma_{j}) (\x )$.  Each of these
expands to a sum of eight products.  Each product consists of three
factors, and the factors are $\PP$'s and $\QQ$'s with various inputs.
In both $f_{1}$ and $f_{2}$ we eliminate a common denominator coming
from $h (\x)$, and then can factor out
$\x^{\sigma_{i}\sigma_{j}\sigma_{i}\beta}$, some common $q$-powers,
and some common monomials coming from the twisting parameter.  After
this we find that the eight remaining terms in $f_{1}$ are
\begin{multline}\label{eq:bigmess}
\PP_{\beta,i} (x_{i}) \PP_{\beta,j} (qx_{i}x_{j})\PP_{\beta ,i}(x_{j}),\quad  
\PP_{\sigma_{i}\bullet\beta,i} (x_{i}) \PP_{\sigma_{i}\bullet\beta,j} (qx_{i}x_{j})\QQ_{\sigma_{i}\bullet \beta ,i}(x_{j}),\\
\PP_{\sigma_{j}\bullet\beta,i} (x_{i}) \QQ_{\sigma_{j}\bullet \beta,j} (qx_{i}x_{j})\PP_{\beta ,i}(x_{j}),\quad 
\PP_{\sigma_{j}\sigma_{i}\bullet\beta,i} (x_{i}) \QQ_{\sigma_{j}\sigma_{i}\bullet\beta,j} (qx_{i}x_{j})\QQ_{\sigma_{i}\bullet\beta ,i}(x_{j}), \\
\QQ_{\sigma_{i}\bullet \beta,i} (x_{i}) \PP_{\beta ,j} (qx_{i}x_{j})\PP_{\beta ,i}(x_{j}),\quad
\QQ_{\beta,i} (x_{i}) \PP_{\sigma_{i}\bullet\beta,j} (qx_{i}x_{j})\QQ_{\sigma_{i}\bullet\beta ,i}(x_{j}), \\
\QQ_{\sigma_{i}\sigma_{j}\bullet\beta,i} (x_{i}) \QQ_{\sigma_{j}\bullet\beta,j} (qx_{i}x_{j})\PP_{\beta ,i}(x_{j}),\\
\QQ_{\sigma_{i}\sigma_{j}\sigma_{i}\bullet\beta,i} (x_{i}) \QQ_{\sigma_{j}\sigma_{i}\sigma_{i}\bullet\beta,j} (qx_{i}x_{j})\QQ_{\sigma_{i}\bullet\beta ,i}(x_{j}),
\end{multline}
and that $f_{2}$ consists of the same eight terms with $i$ and $j$
switched.  Here we again eliminate $\ell$ from the notation.  Note
that each term in $f_{1}$, $f_{2}$ can be uniquely identified by
giving a length three word in $\PP$ and $\QQ$ and saying whether it
appears in $f_{1}$ or $f_{2}$.  The first term in \eqref{eq:bigmess},
for example, can be encoded as $(\PP \PP \PP)_{1}$.  

In summary, to prove $f_{1}=f_{2}$, we need to verify the identity of
rational functions
\begin{equation}\label{eqn:PequalsQ}
(\PP \PP \PP)_{1} + \dotsb + (\QQ \QQ \QQ)_{1} = (\PP \PP \PP)_{2} + \dotsb + (\QQ \QQ \QQ)_{2}. 
\end{equation}
This can be done as follows.  Let $\mu_{i} = \mu_{\ell , i} (\beta)$
and $\mu_{j}=\mu_{\ell ,j} (\beta)$.  Then the identities needed to
check \eqref{eqn:PequalsQ} depend in a minor way on $\mu_{i},\mu_{j}$
mod $n$:
\begin{enumerate}
\item We have $(\PP \PP \PP )_{1}= (\PP \PP \PP)_{2}$, $(\QQ \QQ
\QQ)_{1} = (\QQ \QQ \QQ)_{2}$, $(\PP \QQ \QQ)_{1} = (\QQ \QQ
\PP)_{2}$, and $(\QQ\QQ\PP)_{1} = (\PP \QQ \QQ)_{2}$, independent of 
$\mu_{i}, \mu_{j} \mod n$.
\item If $\mu_{i}$ and $\mu_{j}$ are nonzero mod $n$, then 
$(\QQ \PP \QQ)_{1}= (\QQ \PP \QQ)_{2}$, $(\PP \PP \QQ)_{1} + (\QQ \PP \PP
)_{1} = (\PP \QQ \PP)_{2}$, and $(\PP \QQ \PP)_{1} = (\PP \PP \QQ)_{2} + (\QQ \PP \PP
)_{2}.$
\item \label{thirdcase} If $\mu_{i}=0\mod n$ and $\mu_{j}\not =0\mod
n$, then $(\PP \QQ \PP)_{1} = (\PP \PP \QQ)_{2} + (\QQ \PP \PP)_{2}$
and $(\PP \PP \QQ)_{1} +(\QQ \PP \PP)_{1} +(\QQ \PP \QQ)_{1} = (\PP
\QQ \PP)_{2} +(\QQ \PP \QQ)_{2}$.
\item If $\mu_{i}\not =0\mod n$ and $\mu_{j}=0\mod
n$, then we have the same identities as in case \ref{thirdcase}, but
with $1$ and $2$ switched.
\end{enumerate}
These identities prove \eqref{eqn:PequalsQ}, which implies
$(f|_{\ell}\sigma_{i}\sigma_{j}\sigma_{i} ) (\x ) =
(f|_{\ell}\sigma_{j}\sigma_{i}\sigma_{j}) (\x )$.

Lengthier but entirely similar computations show that the relations
$(\sigma_{i}\sigma_{j})^{4}=1$ and $(\sigma_{1}\sigma_{2})^{6}=1$ (for
$\Phi = G_{2}$) hold, as long as $n$ is relatively prime to all
squared root lengths.  It is also easy to check the required relations
of the form $(\sigma_{i}\sigma_{j})^{2}=1$ if $\alpha_{i}$ and
$\alpha_{j}$ have the same length.

If $n$ is not relatively prime to all squared root lengths, then the
computations are only slightly more complicated.  For instance, assume
that $n = 2m$ is even and that $\sigma_{i},\sigma_{j}$ satisfy
$(\sigma_{i}\sigma_{j})^{4}=1$.  Assume $\norm{\alpha_{i}}^{2}=2$,
$\norm{\alpha_{j}}^{2}=1$, and let $f = \x^{\beta}/h (\x)$ as before.
Apply $\sigma_{i}$ to $f$ as in \eqref{eq:first}. One can then use the
distribution relation
\begin{equation}\label{eq:distrib}
\frac{1}{1-z} = \frac{1}{1-z^{2}} + \frac{z}{1-z^{2}}
\end{equation}
to write each of the $\PP$ and $\QQ$ terms on the right of
\eqref{eq:first}---which have denominator $1- (qx)^{m}/q$---as sums of
two terms with denominator $1- (qx)^{2m}/q^{2}$.  The resulting two
terms from $\PP$ lie in $\local _{\beta} $ and $\local _{\beta
+m\alpha_{i}}$, whereas the two terms from $\QQ$ lie in $\local
_{\sigma_{i}\bullet \beta}$ and $\local _{\sigma_{i}\bullet \beta
+m\alpha_{i}}$ (here the congruence classes should be taken in
$\Lambda /n\Lambda$, not $\Lambda /\Lambda '$).  Hence the computation
is essentially the same as that for $n$ odd, except that each
application of $\sigma_{i}$ to $f$ results in a sum of \emph{four}
terms instead of the two in \eqref{eq:first}.  The computation for
$G_{2}$ is similar, except that one replaces \eqref{eq:distrib} with
an identity with denominator $1-z^{3}$.  Finally, the same trick works
to check the required relations of the form
$(\sigma_{i}\sigma_{j})^{2}=1$ when $\alpha_{i}$ and $\alpha_{j}$ have
different lengths.
\end{proof}

We are now about to prove the main theorem of this section, but before
doing so we require some notation.  Define
\[
\Delta (\x) = \prod_{\alpha >0} (1-q^{m (\alpha)d (\alpha)}\x^{m
(\alpha)\alpha})\quad \text{and}\quad D(\x) = \prod_{\alpha >0}
(1-q^{m (\alpha)d (\alpha)-1}\x^{m (\alpha)\alpha}).
\]
 Put 
\[
j(w,\x)=\Delta(\x)/\Delta(w\x).
\]
This function satisfies the $1$-cocycle relation
\[
j (ww',\x) = j (w,w'\x)j (w',\x)
\]
and can be explicitly computed as follows:

\begin{lemma}\label{lem:jlemma}
We have 
\[
j (w,\x) = \sgn (w)q^{d (\beta)}\x^{\beta },
\]
where 
\[
\beta  = \sum_{\alpha \in \Phi (w)} m (\alpha)\alpha.
\]
\end{lemma}

\begin{proof}
Using \eqref{eqn:monomialaction} we have
\begin{align*}
  \Delta (w\x) &=
\prod_{\alpha >0} (1-q^{m (w^{-1}\alpha ) d (w^{-1}\alpha )}\x^{m
  (w^{-1}\alpha ) w^{-1}\alpha})\\
&=\prod_{\substack{\alpha >0\\w^{-1}\alpha<0}} 
(1-q^{m (w^{-1}\alpha ) d (w^{-1}\alpha )}\x^{m(w^{-1}\alpha )
  w^{-1}\alpha}) \\
&\qquad \times \prod_{\substack{\alpha >0\\w^{-1}\alpha>0}} 
(1-q^{m (w^{-1}\alpha ) d (w^{-1}\alpha )}\x^{m(w^{-1}\alpha )
  w^{-1}\alpha}) \\
& =\prod_{\substack{\alpha <0\\w\alpha>0}} 
(1-q^{m (\alpha ) d (\alpha )}\x^{m(\alpha )\alpha}) 
\prod_{\substack{\alpha >0\\w\alpha>0}} 
(1-q^{m (\alpha ) d (\alpha )}\x^{m(\alpha )\alpha})\\
& =\prod_{\substack{\alpha >0\\w\alpha<0}} 
(1-q^{-m (\alpha ) d (\alpha )}\x^{-m(\alpha )\alpha}) 
\prod_{\substack{\alpha >0\\w\alpha>0}} 
(1-q^{m (\alpha ) d (\alpha )}\x^{m(\alpha )\alpha}),
\end{align*}
where we obtain the third equality above with the substitution
$\alpha\mapsto w^{-1}\alpha$ in both products.
Thus $j(w,\x)$ is equal to the quotient
\begin{align*}
\Delta (\x)/\Bigl(\prod_{\alpha\in \Phi (w)} (1-q^{-m (\alpha ) d
(\alpha)}\x^{-m (\alpha )\alpha})\cdot
\prod_{\substack{\alpha>0\\ \alpha \not \in \Phi (w)}}
(1-q^{m (\alpha ) d (\alpha)}\x^{m (\alpha ) \alpha})\Bigr).
\end{align*}
Multiplying the top and bottom of the last expression by
\[
\prod_{\alpha \in \Phi (w)} -q^{m (\alpha ) d (\alpha)}\x^{m (\alpha
)\alpha} = (-1)^{|\Phi (w)|} \prod_{\alpha \in \Phi (w)}
q^{m (\alpha ) d(\alpha)}\x^{m (\alpha )\alpha},
\]
converts
the denominator to $\Delta (\x)$.  After dividing we find
\[
j (w,\x) = (-1)^{|\Phi (w)|} \prod_{\alpha \in \Phi (w)}
q^{m (\alpha ) d(\alpha)}\x^{m (\alpha )\alpha}.
\]
This agrees with the statement, since $|\Phi (w)| = \length (w)$.  
\end{proof}

We list some further properties of the group action that we will
use.  

\begin{lemma}\label{lemma:actionproperties}
  Let $\ell$ be an $r$-tuple of nonnegative integers.
\begin{enumerate}
\item \label{statement1} Let $f\in \local $. Let $g (\x) \in
\local_\beta$ with $\sigma_i\beta-\beta=0\in\Lambda/\Lambda'.$ Then
\begin{equation}\label{eqn:simplecongruence}
(gf|_{\ell }\sigma_i) (\x ) =g(\sigma_{i}\x)( f|_{\ell }\sigma_i) (\x ). 
\end{equation}
Similarly, if $\beta\in \Lambda'$, then for all $w\in W$ 
\begin{equation}\label{eqn:congruence}
(gf|_{\ell }w) (\x ) =g(w \x)( f|_{\ell }w) (\x).
\end{equation}
\item  \label{statement2} The function $j(w,\x)(1|_\ell w)(\x)$ is regular at the origin.
\item  \label{statement3} Suppose $\length (\sigma_{i}w) = \length (w)+1$.  Then 
\[
j
(\sigma_{i}w,\x) (\x^{(l_{i}+1-m (\alpha_{i}))\alpha_{i}}|_{\ell} w) (\x)
\]
is regular at the origin.
\end{enumerate}
\end{lemma}

\begin{proof}
Since $W$ preserves the sublattice $\Lambda '$, equation
\eqref{eqn:congruence} is an easy consequence of
\eqref{eqn:simplecongruence}, so we prove
\eqref{eqn:simplecongruence}.  It suffices to prove
\eqref{eqn:simplecongruence} for $f (\x) \in \local_{\lambda}.$ Then
$gf$ is still in $\local_{\lambda}$ so \eqref{eqn:simplecongruence}
follows easily from Definition (\ref{wiaction}) of the $|_\ell$
action.  (The point is that $\PP _{\lambda,\ell ,i} (x_{i})$ and $\QQ
_{\sigma_{i}\bullet \lambda,\ell ,i} (x_{i})$ depend only on the
equivalence class of $\lambda$ in $\Lambda/\Lambda'.$)

Statements \eqref{statement2} and \eqref{statement3}
are proved simultaneously by induction.  We will show the details for
the proof of \eqref{statement2}; the proof of \eqref{statement3} is
similar.  

First of all, both statements are obviously true when $\length (w) =
0$.  Now assume $\sigma_{i}w$ satisfies $\length (\sigma_{i}w) =
\length (w)+1$, let $m = m (\alpha_{i})$, and assume that
\eqref{statement2}, \eqref{statement3} are true for $w$.  We prove
that $j (\sigma_{i}w,\x) (1|_{\ell}\sigma_{i} w) (\x)$ is regular.
Indeed, we have
\begin{equation}\label{eqn:prev}
j (\sigma_{i}w,\x) (1|_{\ell}\sigma_{i} w) (\x) = j (\sigma_{i}w,\x)
\bigl ( \bigl[\PP_{0,\ell,i} (x_{i}) + \QQ_{\sigma \bullet 0,\ell ,i}
(x_{i})\bigr]\bigr|_{\ell}w\bigr) (\x),
\end{equation}
and it suffices to check the regularity of the $\PP$ and $\QQ$ terms
separately in \eqref{eqn:prev}.

Now $\PP_{0,\ell,i} (x_{i}) \in \local_{0}$, so we can apply
\eqref{eqn:congruence} with $f=1$, $g=\PP_{0,\ell,i} (x_{i})$.  Thus,
abbreviating $P(\x ) = \PP_{0,\ell,i} (x_{i})$, we have
\[
(P \cdot 1|_{\ell}w) (\x) = P(w\x ) (1|_{\ell }w) (\x). 
\]
Up to irrelevant $q$-powers this expression can be written as
\begin{equation}\label{eqn:starr}
\frac{\x^{k\delta}}{1-q^{m-1}\x^{\delta}} (1|_{\ell}w) (\x),
\end{equation}
where $\delta = m w^{-1} (\alpha_{i})\in \Lambda$ and $k = (l_{i}+1-
(l_{i}+1)_{m})/m$ is nonnegative.  By the description of $\Phi
(\sigma_{i}w)$ in \eqref{eqn:phiwleft}, we have $w^{-1}
(\alpha_{i})>0$.  Therefore the rational function in \eqref{eqn:starr}
is regular at the origin.  By the induction hypothesis $j (w,\x)
(1|_{\ell}w) (\x)$ has no pole at the origin, and so we only need to
check that $j (\sigma_{i}w,\x) (1|_{\ell}w ) (\x)$ has no pole at the
origin.  Looking again at \eqref{eqn:phiwleft}, we find
\[
\sum_{\alpha \in \Phi (\sigma_{i}w)} m (\alpha)\alpha \succ
\sum_{\alpha\in \Phi (w)} m (\alpha)\alpha.
\]
Hence by Lemma \ref{lem:jlemma} the ratio $j (\sigma_{i}w,\x)/j
(w,\x)$ is a monomial of positive degree.  Therefore $j (\sigma_{i}w,\x)
(1|_{\ell}w) (\x)$ has no pole at the origin.  This completes the
analysis of the $\PP$ term.

We can treat the $\QQ$ term in \eqref{eqn:prev} by a similar argument.
We can apply \eqref{eqn:congruence} by taking $f=x_{i}^{l_{i}+1-m}$,
$g = \QQ_{\sigma_{i}\bullet 0, \ell , i} (x_{i})/f$.  Applying
$\phantom{a}|_{\ell}w$ to $gf$, we obtain a rational function with no
pole at the origin times
\[
(x_{i}^{l_{i}+1-m}|_{\ell}w) (\x).
\]
By induction, this becomes regular at the origin after multiplying by
$j (\sigma_{i}w, \x)$.  Therefore this case of \eqref{statement2}
follows from prior cases of \eqref{statement2}, \eqref{statement3}.
Similar computations work to show that \eqref{statement3} follows from
prior cases of \eqref{statement2}, \eqref{statement3}.
\end{proof}


Now define the rational function 
\begin{equation}\label{eqn:defofh}
h (\x ;\ell) = \Delta (\x)^{-1}\sum_{w\in W} j (w,\x) (1|_{\ell}w)
(\x) \in \local.
\end{equation}
The main result of this section is the following theorem:

\begin{theorem}\label{thm:invRatlFunc}  Let $\ell=(l_1, \ldots, l_r)$
  with each $l_i\geq 0.$ The rational function $h = h (\x ;\ell )$ satisfies $h|_\ell w=h$ for all $w\in W.$ 
  Furthermore
$N(\x;\ell)=h(\x;\ell)D(\x)$ is a polynomial in the $x_i$'s.
  Finally, if $\ell=(0,\ldots, 0)$ and
  $m=m (\alpha_{i})$,
  \begin{equation}
    \label{eq:21}
h(0,\ldots, 0, x_i, 0, \ldots, 0;\ell)=
\frac {1+\gamma(\norm {\alpha_i}^2) qx_i}{1-q^{m-1}x_i^m}.
  \end{equation}
\end{theorem}

\begin{proof}
The invariance of $h$ is easy to see.  Indeed, we have 
\[
h (\x ;\ell) = \sum_{w\in W} \frac{1}{\Delta (w\x)} (1|_{\ell}w) (\x).
\]
For any $u\in W$, after applying $\phantom{a}|_{\ell}u$ and using Lemma
\ref{lemma:actionproperties} \eqref{statement1}, we see that
\[
(h|_{\ell}u) (\x) = \sum_{w\in W} \frac{1}{\Delta (wu\x)}(1|_{\ell}wu) (\x).
\]
Thus $h$ is invariant under the $|_{\ell}$-action.

To prove that $N(\x;\ell)$ is a polynomial, introduce for each $w\in
W$ the rational function
$$P_w(\x)=j(w, \x)(1|_\ell w) D_w(\x),$$
where
$$D_w(\x)=\prod_{\alpha \in \Phi (w)} (1-q^{m (\alpha)d (\alpha)-1}
\x^{m(\alpha)\alpha}).
$$
We will show by induction on the length of $w$ that each $P_w(\x)$ is
a polynomial.  If $w$ is the identity there is nothing to prove. 

Suppose that for $w\in W$, the rational function $P_w(\x)$ is a
polynomial.  Let $\sigma_i$ be a simple reflection such that
$\length (w\sigma_i)=\length (w) +1.$ Then
\begin{eqnarray}\nonumber
P_{w\sigma_i}(\x)&=&
j(w\sigma_i,\x)(1|_\ell w\sigma_i)(\x)D_{w\sigma_i}(\x)\\
\nonumber
&=& j(\sigma_i,\x)\left(\left.
\frac {P_w}{D_w}\right|_\ell \sigma_i\right)(\x)\cdot
D_{w\sigma_i}(\x)  \\ 
\label{eqn:jpd}
&=& \frac{ j(\sigma_i,\x)(P_w|_\ell \sigma_i)(\x)}
{D_w(\sigma_i\x)}\cdot
D_{w\sigma_i}(\x), \text{\ \ by Lemma
  \ref{lemma:actionproperties} (\ref{statement1}) }. 
\end{eqnarray}

By the definition \eqref{wiaction} of the action of $\sigma_i,$
we can write $(P_w|_\ell\sigma_i)(\x)$ as
$P_w'(\x)/(1-q^{m(\alpha_i)-1}x_i^{m(\alpha_i)})$ where $P'_w$ is a
Laurent polynomial in the $x_i.$ However, as Lemma
\ref{lemma:actionproperties} (\ref{statement2}) implies that $P_w(\x)$ is regular at
the origin, it follows that $P_w'$ is a polynomial.  Moreover, the
denominator $D_w(\sigma_i\x)$ is equal to
\begin{equation}\label{eqn:fred}
\prod_{\substack{\alpha\in\Phi (w)}} 
(1-q^{m(\alpha)d(\sigma_i\alpha)-1}\x^{m(\alpha)\sigma_i\alpha})=
D_{w\sigma_i}(\x)/(1-q^{m(\alpha_i)-1}x_i^{m(\alpha_i)}),
\end{equation}
where here we use \eqref{eqn:phiwright} to compute $D_{w\sigma_{i}} (\x )$.
Note also that $m(\alpha)=m(w\alpha)$ for all $w\in W, \alpha\in\Phi$,
since the Weyl group preserves root lengths.
Plugging \eqref{eqn:fred} back into \eqref{eqn:jpd}, we conclude that
$$j(w\sigma_i,\x)(1|_\ell w\sigma_i)(\x)D_{w\sigma_i}(\x)= j(\sigma_i,
\x)P_w'(\x)$$ is polynomial.  Therefore
$h(\x; \ell )D(\x)\Delta(\x)=N(\x;\ell)\Delta(\x)$ is a polynomial.

To complete the proof that $N(\x;\ell )$ is a polynomial, we check
that
\begin{equation}\label{eqn:NtimesDelta}
 N (\x;\ell )\Delta (\x) = D (\x)\sum_{w\in W} j(w,\x) (1|_\ell w)(\x)
\end{equation}
is divisible by $\Delta(\x)$.  We begin with the following simple
computation.  Let $B_{\lambda} (\x) = \x^{\lambda}$ be a monomial with
nonnegative exponents.  Let $\sigma_{i}$ be a simple reflection
and let $m=m (\alpha_{i})$.  Assume $\sigma_{i}\lambda
\succ \lambda$.  Then we claim
\begin{equation}\label{eqn:Beqn}
B_{\lambda} (\x) +j (\sigma_{i},\x )(B_{\lambda}|_{\ell}\sigma_{i}
)(\x) 
\end{equation}
can be written as a rational function with numerator
divisible by $1-q^{m}x_{i}^{m}$ and with denominator $1-q^{m-1}x_{i}^{m}$.
Indeed, we have that \eqref{eqn:Beqn} equals
\begin{multline*}
\x^{\lambda}- (qx_{i})^{m}q^{d
(\sigma_{i}\lambda-\lambda)}\x^{\sigma_{i} \lambda} (\PP _{\lambda,\ell ,i} (x_{i}) + \QQ_{\sigma_{i}\bullet \lambda,\ell,i} (x_{i}))=\\
\x^{\lambda} \bigl((1-q^{m-1}x_{i}^{m}) - q^{d (\sigma_{i}\lambda
-\lambda)}\x^{\sigma_{i} \lambda -\lambda}(X + Y)\bigr),
\end{multline*}
where 
\begin{align*}
X &= (qx_{i})^{l_{i}+1+m- (\mu)_{m}} (1-1/q)/ (1-q^{m-1}x_{i}^{m}),\\
Y &= \gamma (qx_{i})^{l_{i}+1} (1-q^{m}x_{i}^{m})/ (1-q^{m-1}x_{i}^{m}),\quad
\gamma \in \C,\\
\mu &= \mu_{\ell ,i} (\sigma_{i}\bullet \lambda).
\end{align*}
It is clear that the $Y$ term has the correct denominator and is
divisible by $1-q^{m}x_{i}^{m}$.  For the $X$ term, after bringing it
together with the initial term we find an expression of the form 
\[
(1-q^{m-1}x_{i}^{m} + q^{km-1}x_{i}^{km} - (qx_{i})^{km})/ (1-q^{m-1}x_{i}^{m}), \quad k\geq 1,
\]
which is a rational function with the required denominator and with
numerator divisible by $1-q^{m}x_{i}^{m}$.

We claim this computation shows that $N(\x;\ell )\Delta (\x)$ is divisible
by $1-q^{m}x_{i}^{m}$.  Indeed, to see this we use the cocycle property of
$j$ to write \eqref{eqn:NtimesDelta} as a double sum, with the inner
sum over a set of minimal length representatives for the right cosets
of $\sigma_{i}$ in $W$, and with the outer sum running over
$1,\sigma_{i}$.  After multiplying by $D (\x)$ the inner sum becomes a
polynomial, and then the computation above shows that after applying the
operator $1+j
(\sigma_{i},\x ) (\phantom{a}|_{\ell}\sigma_{i})$ we obtain a rational
function with denominator killed by $D (\x)$ and with numerator
divisible by $1-q^{m}x_{i}^{m}$.  This also implies
\eqref{eqn:NtimesDelta} is divisible by $1-q^{m (\alpha)}\x^{m
(\alpha)\alpha}$ for any simple root $\alpha$.

We can now show that $1-q^{d(\alpha)m(\alpha)}\x^{m(\alpha)\alpha} $
divides $h (\x ; \ell ) \Delta (\x)$ for all positive roots $\alpha$. Indeed, write
\begin{equation}
   \label{eq:27}
h (\x ; \ell ) \Delta (\x)=(1-q^{m(\alpha_i)}x_i^{m(\alpha_i)})h_{0}(\x),
\end{equation}
say, where $D(\x)h_{0}(\x)$ is a polynomial and the simple root
$\alpha_i$ has the same length as $\alpha.$ Let $w\in W$ map
$\alpha_i$ to $\alpha$. Act on both sides of (\ref{eq:27}) by
$w$. Then $ \pm h(\x;\ell )\Delta (\x )
=(1-q^{d(\alpha)m(\alpha)}\x^{m(\alpha)\alpha} ) (h_{0}|w)$, by Lemma
\ref{lemma:actionproperties} (\ref{statement1}).  But, arguing
inductively as in the first part of the proof of the theorem, $D (\x
)\cdot(h_{0}|w)$ is still polynomial.  We conclude that $h (\x ;
\ell) D (\x)\Delta (\x)$ is divisible by
$1-q^{d(\alpha)m(\alpha)}\x^{m(\alpha)\alpha} $ for every positive
root $\alpha.$ This completes the proof of the polynomiality of $N
(\x;\ell )$.

Now we prove \eqref{eq:21}.  Let $\ell = (0,\dotsc ,0)$ and set all
variables of $h$ equal to zero except for $x_{i}$ in
\eqref{eqn:defofh}.  The sum over $W$ reduces to two terms, namely
$w\in \{1,\sigma_{i} \}$, since all other terms are easily seen to
vanish by Lemma \ref{lem:jlemma}.  Let $m = m (\alpha_{i})$.
Then $\Delta (\x)$ becomes $1- (qx_{i})^{m}$, since this is the only
factor of $\Delta$ involving a monomial in $x_{i}$ alone.  Combining
\eqref{wiaction} and \eqref{eqn:defofh} we find
\begin{equation}\label{eqn:blurgh}
\Delta (\x)^{-1}\bigl(1+j (\sigma_{i} ,\x) (1|_{\ell}\sigma_{i}) (\x)\bigr) =
(1- (qx_{i})^{m})^{-1} (\PP_{0,\ell ,i} + \QQ_{\sigma_{i} \bullet 0, \ell ,i}).
\end{equation}
We have $j (\sigma_{i} ,\x) = - (qx_{i})^{m}$ and $\mu_{\ell ,i} (0) =
-\mu_{\ell ,i} (\alpha_{i})=1$.  Thus the right of \eqref{eqn:blurgh}
becomes
\[
\frac{1}{1- (qx_{i})^{m}}\left(1+\frac{- (qx_{i})^{m} (1-1/q)+(qx_{i}) (1-
(qx_{i})^{m})\gamma (\norm{\alpha_{i}}^{2})}{1- (qx_{i})^{m}/q} \right),
\]
which after a short computation is easily seen to be the right of
\eqref{eq:21}.  This completes the proof of the theorem.
\end{proof}

Write $N (\x ;\ell) = \sum_{\lambda \in \Lambda}
a_{\lambda}\x^{\lambda}$.  Given any $\beta \in \Lambda$ and a simple
root $\alpha_{i}$, we define 
\[
S_{\beta} = S_{\beta , i}= \{\beta +km\alpha_{i}\mid k\in \Z \},
\]
where $m=m (\alpha_{i})$.  Define 
\[
N_{\beta ,i} (\x) = \sum_{\lambda \in S_{\beta}}
a_{\lambda}\x^{\lambda}.
\]

Now choose $\beta \in \Lambda$
 and assume $\sigma_{i}\bullet \beta = \beta + k\alpha_{i}$ with $k\geq
0$.  Define $\delta = (k)_{m}.$ 
Define
\[
f_{\beta ,i} (\x) = \begin{cases}
(N_{\beta ,i} (\x) - \gamma(-\delta) (qx_{i})^{m-\delta}
N_{\sigma_{i}\bullet \beta , i} (\x))/ (1-q^{m-1}x_{i}^{m})&\text{if $\delta \not =0$,}\\
N_{\beta ,i} (\x)/(1-q^{m-1}x_{i}^{m})&\text{otherwise}.
\end{cases}
\]
The function $f_{\beta ,i} (\x)$ satisfies the following symmetry
with respect to the reflection $\sigma_{i}$:

\begin{theorem}\label{thm:f.rootform}
We have 
\[
\frac{f_{\beta ,i} (\x)}{f_{\beta ,i} (\sigma_{i}\x)} = \begin{cases}
(qx_{i})^{l_{i}+1-\delta}&\text{if $\delta \not =0$,}\\
(qx_{i})^{l_{i}+1-m}&\text{otherwise.}
\end{cases}
\]
\end{theorem}

\begin{proof}
We prove the statement when $\delta \not =0$; the remaining case is
simpler and requires no new ideas.

We begin by defining 
\begin{align}
F_{\beta ,i} (\x) &= (N_{\beta ,i} (\x) + N_{\sigma_{i}\bullet \beta ,i}
(\x))/ (1-q^{m-1}x_{i}^{m})\\
 &= \Bigl(\sum_{\lambda \in
S_{\beta}}a_{\lambda}\x^{\lambda} + \sum_{\mu \in S_{\sigma_{i}\bullet \beta}}a_{\mu}\x^{\mu}\Bigr)/ (1-q^{m-1}x_{i}^{m}).\label{eqn:eqa}
\end{align}
By the construction of $h$ and Lemma~\ref{lemma:actionproperties},
$F_{\beta,i} (\x)$ is invariant under
$\phantom{a}|_{\ell}\sigma_{i}$.  On the other hand, explicitly
applying $\sigma_{i}$ to $F$ yields
\begin{multline}\label{eqn:eqc}
(F_{\beta ,i} |_{\ell}\sigma_{i}) (\x) = \Bigl(\sum_{\lambda \in
S_{\beta}} a_{\lambda}B_{\lambda} (\sigma_{i}\x) (\PP_{\beta} +
\QQ_{\sigma_{i}\bullet \beta}) \\
+ \sum_{\mu \in S_{\sigma_{i}\bullet \beta}}a_{\mu} B_{\mu}
(\sigma_{i}\x) (\PP_{\sigma_{i}\bullet \beta} + \QQ_{\beta})\Bigr)/ (1-q^{-m-1}x_{i}^{-m}),
\end{multline}
where we have written $B_{\lambda} (\x) = \x^{\lambda}$ and have
eliminated the $\ell ,i$ from the subscripts to $\PP ,\QQ$ to lighten
the notation.  

Now going from $F_{\beta ,i}$ to $f_{\beta ,i}$ is
achieved by multiplying the terms in $F_{\beta ,i}$ in $S_{\sigma_{i}\bullet
\beta}$ by $-\gamma (-\delta) (px_{i})^{m-\delta}$.  Thus from \eqref{eqn:eqa} we have 
\[
f_{\beta ,i} (\x) = \Bigl(\sum_{\lambda \in
S_{\beta}}a_{\lambda}\x^{\lambda} -\gamma (-\delta) (qx_{i})^{m-\delta}\sum_{\mu \in S_{\sigma_{i}\bullet \beta}}a_{\mu}\x^{\mu}\Bigr)/ (1-q^{m-1}x_{i}^{m}),
\]
which implies 
\begin{multline}\label{eqn:eqb}
f_{\beta,i} (\sigma_{i}\x) \\
= \Bigl(\sum_{\lambda \in
S_{\beta}}a_{\lambda}B_{\lambda }(\sigma_{i}\x )-\gamma (-\delta) (qx_{i})^{\delta-m}\sum_{\mu \in S_{\sigma_{i}\bullet \beta}}a_{\mu}B_{\mu} (\sigma_{i}\x )\Bigr)/ (1-q^{-m-1}x_{i}^{-m}).
\end{multline}
But \eqref{eqn:eqc} also equals $F_{\beta ,i} (\x)$, and there the terms in
$S_{\beta}$ (respectively, $S_{\sigma_{i}\bullet\beta }$) are those that
are multiplied by $\PP_{\beta}, \QQ_{\beta}$ (resp.,
$\PP_{\sigma_{i}\bullet \beta}, \QQ_{\sigma_{i}\bullet \beta}$).  Hence 
\begin{multline}\label{eqn:eqd}
f_{\beta ,i} (\x) = \Bigl(\sum_{\lambda \in
S_{\beta}} a_{\lambda}B_{\lambda} (\sigma_{i}\x) \bigl(\PP_{\beta} -\gamma (-\delta) (qx_{i})^{m-\delta }\QQ_{\sigma_{i}\bullet \beta}\bigr) \\
+ \sum_{\mu \in S_{\sigma_{i}\bullet \beta}}a_{\mu} B_{\mu}
(\sigma_{i}\x) \bigl(-\gamma(-\delta) (qx_{i})^{m-\delta}\PP_{\sigma_{i}\bullet \beta} + \QQ_{\beta}\bigr)\Bigr)/ (1-q^{-m-1}x_{i}^{-m})
\end{multline}
Comparing \eqref{eqn:eqb} and \eqref{eqn:eqd}, we see that $f_{\beta
,i} (\x)/f_{\beta,i } (\sigma_{i}\x )$ equals either of
\[
\PP_{\beta}-\gamma(-\delta) (qx_{i})^{m-\delta}\QQ_{\sigma_{i}\bullet
\beta}, \quad (qx_{i})^{2 (m-\delta)}\PP_{\sigma_{i}\bullet \beta} - \QQ_{\beta} (qx_{i})^{m-\delta } / \gamma(-\delta ),
\]
both of which equal $(qx_{i})^{l_{i}+1-\delta}$.  This completes the
proof of the theorem.
\end{proof}

\section{The coefficients $H (\mathbf{c};\mathbf{m} )$}
\label{section:new}

In this section we explain how $N (\x ;\ell)$ will be used
to construct the factor $H (\mathbf{c};\mathbf{m} )$ that will later
be used to define our multiple Dirichlet series \eqref{eq:wmddef}.
Henceforth, the constants $\gamma(i)$ used in the definition of the
group action in the last section will be specialized to be the
modified Gauss sums
\begin{equation}\label{eq:gstar}
  \gamma(i) = \begin{cases}
    g(1,\p; \epsilon^k)/q & \text{if $k$ is not congruent to 0 mod $n$, }\\
    -1 & \text {otherwise.}
\end{cases}
\end{equation}

Denote the $x_1^{\beta_1}\cdots x_r^{\beta_r}$ coefficient of 
the polynomial $N(\x;\ell)$  by 
\begin{equation}\label{eqn:basicH}
H(\p^{\beta_1},
\ldots, \p^{\beta_r}; \p^{l_1}, \ldots ,\p^{l_r}).
\end{equation}
To
complete the definition of $H$ we impose a \emph{twisted
multiplicativity} property on its coefficients.
For fixed $(c_1\cdots c_r, c_1'\cdots c_r')=1$, we put
\begin{equation}
  \label{eq:3}
H(c_1c_1',\ldots, c_r c_r';{\bf m})=\xi ({\bf c}, {\bf c'})
H(c_1,\ldots, c_r;{\bf m} )   H(c_1',\ldots, c_r';{\bf m}),
\end{equation}
where
\begin{equation}
  \label{eq:4}
  \xi ({\bf c}, {\bf c'})=\prod_{i=1}^r  
  \prs{c_i}{c_i'}^{\norm{\alpha_i}^2} 
  \prs{c_i'}{c_i}^{\norm{\alpha_i}^2}
  \prod_{i<j} \prs{c_i}{c_j'}^{2\langle \alpha_i, \alpha_j\rangle} 
  \prod_{i<j} \prs{c_i'}{c_j}^{2\langle \alpha_i, \alpha_j\rangle} .
\end{equation}
We also have the relation
\begin{equation}
  \label{eq:5}
  H(c_1,\ldots , c_r;m_1m_1',\ldots, m_rm_r' )=
  \prod_{j=1}^r \prs{m_j'}{c_j}^{-\norm{\alpha_j}^2}  
  H(c_1,\ldots , c_r;m_1,\ldots, m_r )
\end{equation}
if $(c_1\cdots c_r, m_1'\cdots m_r')=1.$ Hence, using properties
\eqref{eq:3} and \eqref{eq:5} with \eqref{eqn:basicH}, we define
$H(\mathbf{c} ;\bfm)$ for any $r$-tuples of integers $\mathbf{c}, \bfm$ in
$\OO_S$.

For later arguments it is convenient to state
Theorem~\ref{thm:f.rootform} in different notation.
Given an $r$-tuple of nonnegative integers ${\bf k}=(k_1, \ldots,
k_r)$, define one variable polynomials 
\begin{multline}\label{eq:24}
  N^{(\p;\bf k)}(x;\bfm,\alpha_i)\\
=\sum_{j\geq 0}
H(\p^{k_1}, \ldots, \p^{k_{i-1}}, \p^{j m+(k_i)_m}, \p^{k_{i+1}},
\ldots, \p^{k_r};\bfm) x^{j m+(k_i)_{m}},
\end{multline}
where $m = m (\alpha_{i})$.
For each component $m_{i}$ of $\bfm$, let $l_i=\ord_\p m_i.$ Define 
\begin{equation}
  \label{eq:25}\begin{split}
   f^{(\p;\bf k)}(x;\bfm,\alpha_i)&=\frac{N^{(\p;\bf k)}
     (x;\bfm, \alpha_i)} {1-|\p|^{m-1}x^{m}} \\ -\delta^m_{k_i,k_i'}
g(m_i^{-1}\p^{l_i}&,\p;\epsilon^{\norm{\alpha_i}^2 (k_i-k_i')})
q^{(k_i-k_i'-1)_m}
x^{ (k_i-k_i')_m}\frac{N^{(\p;\bf k')}(x;\bfm,\alpha_i)}
{1-|\p|^{m-1}x^{m}},
\end{split}\end{equation}
where ${\bf k'}$ is defined by
$${\bf k'}=\bigl(k_1,\ldots, k_{i-1},-k_{i}+l_i+1-\sum_{j\neq i} k_jc
(j,i), k_{i+1},\ldots, k_r\bigr), $$ and $\delta^m_{i,j}$ is 0 if
$i\equiv j\pmod{m}$ and 1 otherwise. Note that if $\beta \in \Lambda$
is written as $\beta = \sum k_{i}\alpha_{i}$, where the $\alpha_{i}$
are the simple roots, then $\sigma_{i}\bullet \beta$ is $\sum
k'_{i}\alpha_{i}$.  Moreover, if $\mathbf{m}$ consists of pure
$\p$-powers, then the polynomial $N^{(\p; \mathbf{k})}$ (respectively
$f^{(\p; \mathbf{k})}$) coincides with $N_{\beta ,i}$ (resp.,
$f_{\beta ,i}$) from the end of Section \ref{sec:Waction} after
setting $x_{j}=1$ for all $j\not =i$.  Then a mild generalization of
Theorem~\ref{thm:f.rootform} yields the following:
\begin{theorem}
  \label{thm:f}
We have
$$ \frac{f^{(\p;\bf k)}(x;\bfm,\alpha_i)}{f^{(\p;\bf k)}(1/(q^2x);\bfm,\alpha_i)}=\begin{cases}
(qx)^{l_{i}+1-(k_i'-k_i)_m}&\text{if $(k_{i}'-k_{i})_{m}\not =0$,}\\
(qx)^{l_{i}+1-m}&\text{otherwise.}\\
\end{cases}
. $$
\end{theorem}

\section{Kubota's Dirichlet series} \label{sec:kubota}

In this section we define and describe the functional equations of the
Kubota Dirichlet series.  These are Dirichlet series in one complex
variable whose coefficients are $n^{th}$ order Gauss sums.  These
series arise as the Whittaker coefficients of an Eisenstein series on
the $n$-fold metaplectic cover of $GL_2(F).$ For an integer $j$,
$\psia\in \mathcal M_j(\Omega)$ and $a\in \OO_S^\times$, we define
\begin{equation}
  \label{eq:6}
  \D(s,a; \psia,\epsilon^j)=\sum_{0\neq c\in \OO_S/\OO_S^\times}
\frac{g(a,c;\epsilon^j)\psia(c)}{\nrm{c}^s}.
\end{equation}
This is initially defined for $\re(s)>3/2,$ where the series is
absolutely convergent.

Let $m=\gcd(n, j)$ and
set 
\begin{equation}\label{eqn:defofG}
G_m(s)=\bigl((2\pi)^{-(m-1)(s-1)}{\Gamma(ms -m)}/{\Gamma
   (s-1)}\bigr)^{[F:\Q]/2}.
\end{equation}
Define
\begin{equation}
  \label{eq:7}
  \D^*(s,a; \psia,\epsilon^j)=G_m(s)\zeta(ms-m+1)
\D(s,a; \psia,\epsilon^j)
\end{equation}
where $\zeta$ is the Dedekind zeta function of $F.$

If $\psia \in {\mathcal M}_j(\Omega)$ and
$\eta \in F^\times$ we define
\begin{equation}\label{def:Psieta}
  \hat\psia_\eta(c)=(\eta, c)^j\psia(\eta c)\text{\ \ \ and\ \ \  }
  \tilde\psia_\eta(c)=(\eta, c)^j\psia(\eta^{-1} c^{-1}).
\end{equation}
It is easy to see that $\hat\psia_\eta$ and
$\tilde\psia_\eta$ are both in ${\mathcal
  M}_j(\Omega)$ and that they depend only on the class of $\eta$ in
$F^\times_S/F^{\times n}_S$, where the $n$ denotes taking $n$th powers.  
Then we have the following:
\begin{prp}[\!\!\cite{ku1,ku2,bb}]\label{prp:KubotaFE}
  The function $ \D^*(s,a; \psia,\epsilon^j)$ has a meromorphic
  continuation to $\C$ and is holomorphic except for possible simple
  poles at $s=1\pm1/m.$ Moreover $\D^*(s,a; \psia,\epsilon^j)$
  satisfies the functional equation
$$\D^*(s,a; \psia,\epsilon^j)=\nrm {a}^{1-s} 
\sum_{\eta\in F_S^\times/F_S^{\times n}} P_{a \eta}^j(s) 
\D^*(2-s, a; \tilde\psia_\eta, \epsilon^j).$$   
Here the $ P_{a \eta}^j(s)$ are Dirichlet polynomials supported on the
places in $S.$ 
\end{prp}

\begin{remark}
Based on the explicit functional equations given in the function field
case (see \cite{sjpNote,jhoff}), we expect the Dirichlet polynomials
to be closely related to the rational functions $\PP$ and $\QQ$ which
appear in Section \ref{sec:Waction}.  However, the nature of the
precise relationship is unclear and not needed for what follows.
\end{remark}

Given a set of primes $T,$ we define
\begin{equation}
  \D_T(s,a;\psia, \epsilon^j)=\sum_{\substack{0\neq c\in
      \OO_S/\OO_S^\times\\(c, T)=1 } }
  \frac{g(a,c;\epsilon^j)\psia(c)}{\nrm{c}^s}.
\end{equation}
If $\m_0=\prod_{\p\in T} \p$ we sometimes write $\D_{\m_0}(s, a;\psia,
\epsilon^j) $  for $\D_T(s,a;\psia, \epsilon^j).$

Using properties of Gauss sums, we can relate the functions $\D_T$ to
the functions $\D_{T'}$ for different sets $T$ and $T'.$ This is the
content of the following two lemmas.

\begin{lemma} Let $\p\in \OO_S/\OO_S^\times$ be prime of norm $q.$ 
For an integer
$i$ with $0\leq i\leq m-1$ and $a_1,a_2,\p$ all pairwise relatively
prime, we have
\begin{equation*}\begin{split}
\D_{a_1}(s,a_2\p^i;\psia, \epsilon^j)\qquad\qquad\qquad\qquad\qquad
\qquad\qquad\qquad\qquad\qquad\\
=\D_{\p a_1}(s, a_2\p^i;\psia, \epsilon^j) + 
\frac{g(a_2\p^i, \p^{i+1};\epsilon^j)} {q^{(i+1)s}} 
\D_{\p a_1}(s,a_2\p^{(m-i-2)_m};\hat\psia_\eta, \epsilon^j)
\end{split}\end{equation*}
where $\eta\sim \p^{i+1}$.  Here we write $a\sim b$ to mean that $a,b$
lie in the same coset modulo $F_{S}^{\times n}$.

\end{lemma}

\begin{proof}
For $\p, a_1, a_2$ as in the statement, 
\begin{align*}
 \D_{a_1}&(s,a_2\p^i;\psia, \epsilon^j)\\
&= \sum_{\substack {0\neq c\in
    \OO_S/\OO_S^\times\\(c, a_1)=1}}
\frac{g(a_2\p^i,c;\epsilon^j)\psia(c)} {|c|^s} \\
&=\sum_{k\geq 0} \sum_{\substack{(c, a_1\p)=1}}
\frac{g(a_2\p^i,c\p^k;\epsilon^j)\psia(q^k c)}{|c|^sq^{ks}}\\
&=\sum_{k\geq 0} \sum_{\substack{(c, a_1\p)=1 }}
\frac{g(a_2\p^i,c;\epsilon^j)g(a_2\p^i, \p^k;\epsilon^j)\psia(q^k c)}
{|c|^sq^{ks}}\prs{c}{\p^{k}}^j\prs{\p^k}{c}^j \\
&= \sum_{\substack{(c, a_1\p)=1 }}
\frac{g(a_2\p^i,c;\epsilon^j)}{|c|^s}\biggl( \sum_{k\geq 0}
\frac{g(a_2\p^i, \p^k;\epsilon^j)\psia(q^k c)}{q^{ks}} 
\prs{\p^{2jk}}{c}(\p^k, c)_S^j\biggr) 
\end{align*}
The Gauss sum in the inner sum vanishes unless $k=0$ or $i+1.$  This
proves the lemma.
\end{proof}

Inverting the previous lemma, we obtain

\begin{lemma}\label{lemma:dp}
If $0\leq i\leq m-2$ and $a_1, a_2, \p$ are as above, and 
$\eta\sim \p^{i+1}$, then  
\begin{align*}\D_{\p a_1}&(s,a_2 \p^i;\psia,\epsilon^j)= \\
&\frac{\D_{a_1}(s,a_2\p^i;\psia,\epsilon^j)}{1-|\p|^{m-1-ms}}- 
\frac{g(a_2\p^i,\p^{i+1};\epsilon^{j})}{|\p|^{(i+1)s}} 
\frac{\D_{a_1}(s,a_2\p^{m-i-2};\hat\psia_\eta,\epsilon^j )}
{1-|\p|^{m-1-ms}}, \end{align*}
and if $i=m-1$, then
$$\D_{\p a_1}(s, a_2\p^i;\psia, \epsilon^j)=
\frac{\D_{a_1}(s,a_2\p^i;\psia, \epsilon^j)}{1-|\p|^{m-1-ms}}.$$ 
\end{lemma}

Before stating and proving the main result of this section, we need to
extend the functional equation of the Kubota series $\D(s,a;\psia,
\epsilon^j)$ to a slightly more general class of $\psia.$ We follow
Section 5 of \cite{wmd2} and Section 7 of \cite{wmd4}.  Let
$\bfm=(m_1, \ldots, m_r).$ This $r$-tuple will be fixed for the rest
of the section.  Let $\M (\Omega,\Phi)$ be the space of functions
$\psia\colon (F_S^\times)^r\to \C$ such that
\begin{equation}
  \label{eq:18}
  \psia(\varepsilon_1 a_1, \ldots ,\varepsilon_r a_r)=\prod_{i=1}^r
(\varepsilon_i, a_i)_S^{\norm{\alpha_i}^2}\Bigl(\prod_{i<j} 
(\varepsilon_i, a_j)_S^{2\langle \alpha_i, \alpha_j\rangle}\Bigr)
  \psia(a_1, \ldots , a_r).
\end{equation}
 Let $\mathcal{A}$ be the
ring of Laurent polynomials in $|\p_v|^{s_i}$ where $v$ runs over the
places in $S_{\fin}.$  Define $\MM=\MM(\Omega, \Phi)=
\mathcal{A}\otimes\M (\Omega,\Phi)$  and $\MM_j(\omega)=
\mathcal{A}\otimes\M_j (\Omega).$
If $\psia\in \MM(\Omega,\Phi)$
and $a_1, \ldots, a_r\in  \OO_S/\OO_S^\times,$  define
 $\psia_i^{(a_1, \ldots, a_r)}$ by
\begin{equation}
  \label{eq:10}
  \psia_i^{(a_1, \ldots, a_r)}(\s;c)=
\psia(\s;a_1, \ldots, a_i c, \ldots,  a_r) 
(c, a_i)_S^{-\norm{\alpha_i}^2} 
\prod_{j>i} (c,a_j)_S^{-2\langle \alpha_i, \alpha_j\rangle}. 
\end{equation}

\begin{lemma}[Lemma 5.3 \cite{wmd2}]
For fixed  $a_1, \ldots, a_r\in  \OO_S/\OO_S^\times$, we have
\[
\psia_i^{(a_1, \ldots, a_r)}\in \MM_{\norm{\alpha_i}^2}(\Omega). 
\]
\end{lemma}

Continuing to follow \cite{wmd2, wmd4} we define an action of the Weyl
group $W$ on $\MM (\Omega,\Phi).$  First, we let the simple reflection
$\sigma_{i}$ act on the complex vector $\s$ by 
\[
(\sigma_{i}\s)_{j} = s_{j}-c (j,i) (s_{i}-1),  
\]
where $c (j,i)$ is the Cartan integer. This corresponds to the action
\eqref{eqn:wiaction1} under the change of variables $x_{i} =
q^{-s_{i}}$. 
Then for the simple reflection $\sigma_i$
and $\psia\in \MM (\Omega,\Phi)$, define
\begin{equation}
  \label{eq:19}\begin{split}
  (\sigma_i \psia)(\s;a_1, \ldots , a_r)=&
\Bigl(\sum_{\eta\in F_S^\times/F_S^{\times n}}
P^{\norm{\alpha_i}^2}_{\eta m_ib_i}(s_i) (\eta,
a_i)_S^{\norm{\alpha_i}^2} \prod_{j>i} 
(\eta, a_j)_S^{2\langle \alpha_i, \alpha_j\rangle}\Bigr)\\
& \ \ \times\psia(\sigma_i\s;a_1, \ldots , a_r)
\end{split}\end{equation}
where $b_i=\prod_i a_i^{-2\langle \alpha_i, \alpha_j\rangle/ \langle
\alpha_i, \alpha_i\rangle}$ and $\s=(s_1, \ldots, s_r).$ The
$W$-action on $\MM (\Omega,\Phi)$ depends on the twisting parameter
$\bfm$, but we suppress this dependence from the notation.  When there
is no chance of confusion, we will not explicitly indicate the $\s$
dependence for $\psia$ in $\MM (\Omega,\Phi)$ or $\MM_j (\Omega).$
\begin{prp}[Proposition 5.6 \cite{wmd2}]
If $\psia\in\MM (\Omega,\Phi)$ then $\sigma_i\psia \in
\MM (\Omega,\Phi).$
\end{prp}

We can now restate the functional equation of Proposition
\ref{prp:KubotaFE} in a slightly cleaner form.

\begin{prp}[Lemma 5 \cite{wmd4}]\label{prp:KubotaFE2}
  Given $\psia\in \MM(\Omega, \Phi)$ we have
\begin{align*}\D^*(s_i, m_ib_i;&\psia_i^{(a_1,\ldots, a_r)}, 
\epsilon^{\norm{\alpha_i}^2}) \\
&=\nrm{m_ib_i}^{1-s_i}
\D^*(2-s_i, m_ib_i; (\sigma_i\psia)_i^{(a_1,\ldots, a_r)}, 
\epsilon^{\norm{\alpha_i}^2})
\end{align*}
for $b_i=\prod_i a_i^{-2\langle \alpha_i, \alpha_j\rangle/ 
\langle \alpha_i, \alpha_i\rangle}.$
\end{prp}

We now turn to the main result of this section.  Let ${\bf a}$ be the
$(r-1)$-tuple $ (a_1,\ldots, \hat a_i, \ldots,a_r)$, where the hat on 
$a_i$ indicates that this entry is omitted. Let
$\psia\in \MM(\Omega, \Phi).$ Introduce the Dirichlet series
\begin{equation}
  \label{eq:8}\begin{split}
  \E(s_i,{\bf a};\bfm, \psia,i)= 
\sum_{0\neq a_i\in \OO_S/\OO_S^\times}
\frac{H(a_1,\ldots,  a_i, \ldots,a_r ;
{\bf m})\psia(a_1,\ldots,  a_i, \ldots,a_r)}
{\nrm{a_i}^{s_i}}.
\end{split}
\end{equation}
Let $m=m (\alpha)$
and define 
\begin{equation}
  \label{eq:9}
\E^*(s_i,  {\bf a};{\bf m},\psia,i)=\zeta(ms_i-m+1)
G_m(s_i)\E(s_i, {\bf a};{\bf m},\psia,i).
\end{equation}

\begin{theorem}\label{thm:FEa}
  Let $A=\prod_{j\neq i} a_j^{-2\langle \alpha_i,
    \alpha_j\rangle/\norm{\alpha_i}^2}.$ 
Then 
$$\E^*(s_i,  {\bf a};{\bf m},\psia,i)=\nrm {Am_i}^{1-s_i} 
\E^*(2-s_i,  {\bf a};{\bf m},\sigma_i\psia,i).$$   
\end{theorem}

\begin{proof}
To simplify notation, assume that $i=1.$ Let $\p_1,\ldots,
  \p_V$ be the prime divisors of $a_2\cdots a_rm_1\cdots m_r,$ with
  $\nrm{\p_j}=q_j.$ Let $T=\{\p_1,\ldots, \p_V\}.$ Write
  $a_j=\p_1^{\beta_{j1}}\cdots \p_V^{\beta_{jV}}$ for $j=2,\ldots r$
  and $Am_1=\p_1^{l_1}\ldots\p_V^{l_V}.$ If $a_1$ has no divisors 
in $T$  and $a_1'=\p_1^{\beta_{11}}\cdots \p_V^{\beta_{1V}}$, we 
expand the  $H$-coefficient
\begin{align*}
H&(a_1a_1',
a_2, \ldots, a_r;\bfm)=
 \prs{a_1'}{a_1}^{2\norm{\alpha_1}^2}
\prs{A}{a_1}^{-\norm{\alpha_1}^2}(a_1',a_1)_S^{\norm{\alpha_1}^2}
(A,a_1)_S^{-\norm{\alpha_1}^2}\\
&\times H(a_1, 1,\ldots, 1;\bfm)
H(\p_1^{\beta_{11}}\cdots \p_V^{\beta_{1V}},
\p_1^{\beta_{21}}\cdots \p_V^{\beta_{2V}},\ldots ,
\p_1^{\beta_{r1}}\cdots \p_V^{\beta_{rV}};\bfm).  
\end{align*}
By the multiplicativity relation (\ref{eq:5}) and Theorem
\ref{thm:invRatlFunc},  the first $H$-coefficient on the 
right-hand side is 
$$H(a_1, 1,\ldots, 1;\bfm)=g(1,a_1;\epsilon^{\norm{\alpha_1}^2})
\prs{m_1}{a_1}^{-\norm{\alpha_1}^2}.$$
We further expand the second $H$-coefficient on the right-hand side 
as  
\begin{align}
  \label{eq:12}
  H(&\p_1^{\beta_{11}}\cdots  \p_V^{\beta_{1V}},
\p_1^{\beta_{21}}\cdots \p_V^{\beta_{2V}},\ldots ,
\p_1^{\beta_{r1}}\cdots \p_V^{\beta_{rV}};\bfm)\\ \nn
&=
\Bigl(\prod_{\substack{1\leq j\leq r\\i\neq i'}}
  \prs{\p_i^{\beta_{ji}}} {\p_{i'}^{\beta_{ji'}}}^{\norm{\alpha_j}^2}
\Bigr)
\Bigl(\prod_{\substack{j<j'\\i\neq i'}} \prs{\p_i^{\beta_{ji}}}
{\p_{i'}^{\beta_{j'i'}}}^{2\langle \alpha_j,\alpha_{j'}\rangle}\Bigr)
\prod_{i=1}^V H(\p_i^{\beta_{1i}}, \ldots,\p_i^{\beta_{ri}};\bfm)
\end{align}

Therefore, up to a multiplicative factor of norm 1 coming from a
product of  power residue symbols, $ \E(s, {\bf
  a};{\bf m},\psia,1) $ is equal to 
\begin{align}
  \label{eq:13}
\sum_{\substack{(a, T)=1\\k_1, \ldots, k_V\geq 0}}
\frac{g(1,a;\epsilon^{\norm{\alpha_1}^2})}
{\nrm{a}^{s_1} q_1^{k_1s_1}\cdots q_V^{k_Vs_1}} 
\prs{\p_1^{2k_1}\cdots\p_V^{2k_V}A^{-1}m_1^{-1} }
{a}^{\norm{\alpha_1}^2}
\psia_1^{(\p_1^{k_1}\cdots\p_V^{k_V},a_2,\dots, a_r)}(a_1)
\\ \nn  \times
\Bigl(\prod_{\substack{i\neq i'}}
 \prs{\p_i^{k_{i}}} {\p_{i'}^{k_{i'}}}^{\norm{\alpha_1}^2}\Bigr)
\Bigl(\prod_{\substack{2\leq j\leq r\\i\neq i'}} 
\prs{\p_i^{k_{i}}}
{\p_{i'}^{\beta_{ji'}}}^{2\langle \alpha_1,\alpha_j\rangle}\Bigr)
\prod_{i=1}^V H(\p_i^{k_{i}},\p_i^{\beta_{2i}},
\ldots,\p_i^{\beta_{ri}};\bfm)\\
\nn   =  \sum_{k_1, \ldots, k_V=0}^{m-1} 
\D_T(s_1,\p_1^{(l_1-2k_1)_m}\cdots \p_V^{(l_V-2k_V)_m};
\psia_1^{(\p_1^{k_1}\cdots\p_V^{k_V},a_2,\dots, a_r)},
\epsilon^{\norm{\alpha_1}^2})\\
\nn  \times\prod_{i=1}^V N^{(\p_i;k_i,\beta_{2i},\ldots,\beta_{ri})}
(q_i^{-s_1}; \bfm,\alpha_1)
\Bigl(\prod_{\substack{i\neq i'}}
  \prs{\p_i^{k_{i}}} {\p_{i'}^{k_{i'}}}^{\norm{\alpha_1}^2}
\Bigr)
\Bigl(\prod_{\substack{2\leq j\leq r\\i\neq i'}} 
\prs{\p_i^{k_{i}}}
{\p_{i'}^{\beta_{ji'}}}^{2\langle \alpha_1,\alpha_j\rangle}\Bigr).
\end{align}

Let $\theta^{(\p_1^{k_{1}})}=\psia_1^{(\p_1^{k_1}\cdots\p_V^{k_V},a_2,\dots,
  a_r)} $
and denote by $C(k_1)=C(k_1, \ldots, k_r)$ the product of
residue symbols
\begin{equation}
  \Bigl(\prod_{\substack{i\neq i'}}
  \prs{\p_i^{k_{i}}} {\p_{i'}^{k_{i'}}}^{\norm{\alpha_1}^2}
\Bigr)
\Bigl(\prod_{\substack{2\leq j\leq r\\i\neq i'}} 
\prs{\p_i^{k_{i}}}
{\p_{i'}^{\beta_{ji'}}}^{2\langle \alpha_1,\alpha_j\rangle}\Bigr).
\end{equation}
Letting $K_i=(l_i-2k_i)_m$ for $i=i, \ldots, r$, we have by Lemma
\ref{lemma:dp} 
\begin{align}\label{eq:usedLemma412}
 &(1-|\p_1|^{m-1-ms_1}) \D_T(s_1,\p_1^{K_1}\cdots \p_V^{K_V};
\theta^{(\p_1^{k_1})},\epsilon^{\norm{\alpha_1}^2})C(k_1) =\\ \nn
& \D_{T-\{\p_1\}}(s_1, \p_1^{K_1}\cdots \p_V^{K_V};
\theta^{(\p_1^{k_1})},\epsilon^{\norm{\alpha_1}^2})C(k_1) 
-\delta^m_{K_1,-1}\frac{g(\p_1^{K_1}\cdots \p_V^{K_V}, \p_1^{K_1+1}; \epsilon^{\norm{\alpha_{1}}^{2}})}
{q_1^{(K_1+1)s_1}} \\ \nn
& \ \times \D_{T-\{\p_1\}}(s_1, \p_1^{(2k_1-l_1-2)_m}\cdots \p_V^{K_V};
\hat\theta_\eta^{(\p_1^{k_1})},\epsilon^{\norm{\alpha_1}^2})C(k_1),
\end{align}
where $\eta\sim\p_1^{K_1+1}$ and $\delta$ is defined after \eqref{eq:25}.  
Consider the second term on the right, with
$k_1$ replaced by $(l_1+1-k_1)_m$.  
For $\delta^m_{K_1,-1}\neq 0$ this gives  
\begin{equation}
  \label{eq:314}
  \begin{split}
  \frac{g(\p_1^{(2k_1-l_1-2)_m}\p_2^{K_2}\cdots \p_V^{K_V}, 
\p_1^{(2k_1-l_1-1)_m}; \epsilon^{\norm{\alpha_{1}}^{2}})}
{q_1^{(2k_1-l_1-1)_ms_1}} 
C(l_1-k_1+1) \\
\times \D_{T-\{\p_1\}}(s_1,\p_1^{K_1}\cdots \p_V^{K_V};
\hat\theta_{\eta'}^{(\p_1^{(l_1+1-k_1)_m})},
\epsilon^{\norm{\alpha_1}^2})
  \end{split}
\end{equation}
where $\eta'\sim\p_1^{2k_1-l_1-1}.$

The Gauss sum can be written as 
\begin{equation}
  \label{eq:315}
 \prs{\p_2^{K_2}\cdots \p_V^{K_V}}{\p_1^{2k_1-l_1-1}}^{-\norm{\alpha_1}^2} 
g(\p_1^{(2k_1-l_1-2)_m}, \p_1^{(2k_1-l_1-1)_m}; \epsilon^{\norm{\alpha_{1}}^{2}}).
\end{equation}
Keeping careful track of the Hilbert symbols, we get
\begin{lemma}\label{lemma:Ck1}
For $\eta'$ as above,
$$C(k_1)\theta^{(\p_1^{k_1})}=C(l_1-k_1+1)
\prs{m_1\p_2^{K_2}\cdots
  \p_V^{K_V}}{\p_1^{2k_1-l_1-1}}^{-\norm{\alpha_1}^2 }
\hat\theta^{(\p_1^{(l_1+1-k_1)_m})}_{\eta'}.$$
\end{lemma}

\begin{proof}[Proof of Lemma \ref{lemma:Ck1}] Define $l_i'$ by
$\prod_{2\leq j\leq r} \p_i^{-2\langle \alpha_1, \alpha_j\rangle
\beta_{ji}}=\p_i^{l_i'}.$ We have
\begin{align}
  \label{eq:17}
  \frac{C(k_1)}{C(l_1-k_1+1)}&=\frac{
\prs{\p_1^{k_1}} {\p_2^{k_2}\cdots \p_V^{k_V}}^{\norm{\alpha_1}^2} 
\prs {\p_2^{k_2}\cdots \p_V^{k_V}}{\p_1^{k_1}}^{\norm{\alpha_1}^2} 
\prs{\p_1^{k_1}} {\p_2^{l_2'}\cdots \p_V^{l_V'}}^{-\norm{\alpha_1}^2} 
}
{\prs{\p_1^{l_1-k_1+1}} 
  {\p_2^{k_2}\cdots \p_V^{k_V}}^{\norm{\alpha_1}^2} 
\prs {\p_2^{k_2}\cdots \p_V^{k_V}}
   {\p_1^{l_1-k_1+1}}^{\norm{\alpha_1}^2} 
\prs{\p_1^{l_1-k_1+1}} 
   {\p_2^{l_2'}\cdots \p_V^{l_V'}}^{-\norm{\alpha_1}^2} } 
\\ \nn
&\hspace{-.4in}= \prs {\p_1^{2k_1-l_1-1}}
{\p_2^{2k_2-l_2'}\cdots  \p_V^{2k_V-l_V'}}^{\norm{\alpha_1}^2}
  (\p_1^{2k_1-l_1-1}, \p_2^{k_2}
\cdots \p_V^{k_V})^{\norm{\alpha_1}^2}\\ \nn
&\hspace{-.4in}= \prs{\p_2^{2k_2-l_2'}\cdots  \p_V^{2k_V-l_V'}}
{\p_1^{2k_1-l_1-1}}^{\norm{\alpha_1}^2} 
(\p_1^{2k_1-l_1-1}, \p_2^{l_2'-k_2}
\cdots \p_V^{l_V'-k_V})^{\norm{\alpha_1}^2}
\end{align}
Further, $\theta^{(\p_1^{k_1})}/
\hat\theta^{(\p_1^{(l_1+1-k_1)_m})}_{\eta'}=
(\p_1^{2k_1-l_1-1},\p_2^{k_2-l_2'}\cdots\p_V^{k_V-l_V'})^{
\norm{\alpha_1}^2}.$
Thus 
\begin{equation}
  \label{eq:22}
  \begin{split}
 \frac{C(k_1)\theta^{(\p_1^{k_1})}}
{C(l_1-k_1+1)\hat\theta^{(\p_1^{(l_1+1-k_1)_m})}_{\eta'}}
=& \prs{\p_2^{2k_2-l_2'}\cdots  \p_V^{2k_V-l_V'}}
{\p_1^{2k_1-l_1-1}}^{\norm{\alpha_1}^2}\\ 
=&
\prs{m_1\p_2^{K_2}\cdots  \p_V^{K_V}}
{\p_1^{2k_1-l_1-1}}^{-\norm{\alpha_1}^2}
.   
  \end{split}
\end{equation}

\end{proof}

We now continue from the last line of (\ref{eq:13}).  Use
(\ref{eq:usedLemma412}) to write $\D_T$ as a linear combination of two
Kubota series of type $\D_{T-\{\p_1\}}$ and replace $k_1$ by
$(l_1+1-k_1)_m$ in the second of these.  Then we use Lemma
\ref{lemma:Ck1} to put the two Kubota series back together and find
that, up to a product of residue symbols, $ \E(s_1, {\bf a};{\bf
  m},\psia,1) $ is
\begin{align}
  \label{eq:11}
     &\hspace{-.4in}  \sum_{k_1, \ldots, k_V=0}^{m-1} 
\D_{T-\p_1}(s_1, \p_1^{(l_l-2k_1)_m}\cdots \p_V^{(l_V-2k_V)_m};
\theta^{(\p_1^k)},
\epsilon^{\norm{\alpha_1}^2})C(k_1,\ldots, k_r)\\
\nn &\times f^{(\p_1;k_1,\beta_{21},\ldots,\beta_{r1})}
(q_1^{-s_1}; \bfm, \alpha_1)  \prod_{i=2}^V
N^{(\p_i;k_i,\beta_{2i},\ldots,\beta_{ri})} (q_i^{-s_1};\bfm,
\alpha_1).
\end{align}
Repeating this process for $\p_2, \ldots, \p_r$ in order to remove the
remaining primes from $T,$ we arrive at
\begin{align}
  \label{eq:15}
  \E(s_1, {\bf a};{\bf m},&\psia,1)  \\
= \xi& \sum_{k_1, \ldots, k_V=0}^{m-1} 
\D(s_1,\p_1^{(l_1-2k_1)_m}\cdots \p_V^{(l_V-2k_V)_m};
\psia_1^{(\p_1^{k_1}\cdots\p_V^{k_V},a_2,\dots, a_r)},
\epsilon^{\norm{\alpha_1}^2})\\
\nn &\times C(k_1,\ldots, k_r)\prod_{i=1}^V
f^{(\p_i;k_i,\beta_{2i},\ldots,\beta_{ri})} (q_i^{-s_1};\bfm ,\alpha_1),
\end{align}
for $\xi$ a product of residue symbols.  The theorem is now a
consequence of the functional equation of Kubota's Dirichlet series
given in Proposition \ref{prp:KubotaFE2}, together with the functional
equation of $f$ given in Theorem \ref{thm:f}.
\end{proof}

\section{The multiple Dirichlet series}\label{sec:wmds}

Given $\psia\in  \MM(\Omega, \Phi)$ and ${\bf m}=(m_1, \ldots m_r)$ 
an $r$-tuple of nonzero integers in $\OO_S,$ we define the  
multiple Dirichlet series in $r$ complex variables
\begin{equation}
  \label{eq:wmddef}
  Z(\s;{\bf m},\psia)=\sum_{0\neq c_1,\ldots, c_r\in \OO_S/ \OO_{S}^{\times }} 
\frac{H(c_1,\ldots, c_r;{\bf m})
    \psia(c_1,\ldots, c_r)} {\prod |c_i|^{s_i}} .
\end{equation}
This is initially defined for $\s=(s_1, \ldots , s_n)$ an $r$-tuple 
of complex numbers with $\re(s_i)> 3/2.$   
To describe the functional equations satisfied by $Z$, we need to
introduce some Gamma factors.  For any positive root $\alpha \in
\Phi^{+}$, write $\alpha = \sum k_{i}\alpha_{i}$ as in
\eqref{eqn:repn}.  Define 
\[
\zeta_{\alpha} (s) = \zeta \bigl(1+m (\alpha)\sum_{i=1}^{r} k_{i} (s_{i}-1)\bigr),
\]
where $\zeta$ is the Dedekind zeta function of $F$, and define 
\[
G_{\alpha} (s) = G_{m (\alpha)}\bigl(1/2+\sum_{i=1}^{r}k_{i} (s_{i}-1)/2\bigr),
\]
where $G_{m}(s)$ is defined in \eqref{eqn:defofG}. 
Finally we put 
\[
Z^*(\s;\bfm,\psia) = Z(\s;\bfm,\psia)\prod_{\alpha >0} G_{\alpha} (s)\zeta_{\alpha} (s),
\]
and we can now state our main result:

\begin{theorem}\label{thm:maintheorem}
Let $\psia\in \MM(\Omega, \Phi).$  The function $Z(\s;\bfm,\psia)$ has
a meromorphic continuation to $\s\in\C^r.$  Moreover
$Z(\s;\bfm,\psia)$ satisfies a group of functional equations under
$W,$ the Weyl group of $\Phi.$  For the simple reflection $\sigma_i$ we
have
\begin{equation}
  \label{eq:20}
  Z^*(\s;\bfm,\psia)=\nrm{m_i}^{1-s_i}
Z^*(\sigma_i\s;\bfm,\sigma_i\psia). 
\end{equation}
The set of  polar hyperplanes is  contained in the 
 $W$-translates of the hyperplanes 
$s_i=1\pm 1/\gcd(n,\norm{\alpha_i}^{2}).$ 
\end{theorem}

\begin{proof}
We will show that the functional equations are valid, whenever both
sides of (\ref{eq:20}) are defined. Given this, the meromorphic
continuation of $Z^*$ and the validity of the functional equations for
all $\s\in\C^r$ away from the polar hyperplanes is a routine
consequence of the Bochner's tube principle as in \cite{qmds} or
\cite{wmd2}.  We refer the reader to these papers for the details.

To prove the $\sigma_i$ functional equation, fix $c_{j}$ for $i\not
=j$ and consider the sum over the
$i^{th}$ index in the series defining $Z(\s;\bfm,\psia)$: 
\begin{equation}
  \label{eq:26}
  \sum_{0\neq c_i\in \OO_S/\OO_{S}^{\times }} 
\frac{H(c_1,\ldots, c_r;{\bf m})
\psia(c_1,\ldots, c_r)} {|c_i|^{s_i}}.
\end{equation}
This is the function $\E(s_i,{\bf c};\bfm, \psia,i)$ of \eqref{eq:8},
which satisfies the functional equation of Theorem \ref{thm:FEa}. The
functional equation of $Z^*(\s;\bfm,\psia)$ under $\sigma_i$ is now
immediate.
\end{proof}
 
\section{Examples and comments}\label{sec:examples}

We list some examples of the Weyl group multiple Dirichlet series
constructed in this paper.  In this section, we will simplify the
exposition by ignoring such complications as the exact form of
reciprocity, Hilbert symbols, and the use of $S$-integers instead of
integers.

\subsection{Quadratic series} In \cite{qmds} we gave a list of
examples of quadratic Weyl group multiple Dirichlet series associated
to simply-laced root system.  In summary, the $A_2, A_3$ and $D_4$
series can be used to study the first, second and third moments of
quadratic Dirichlet $L$-functions respectively, as in
Goldfeld--Hoffstein \cite{gh}, Fisher--Friedberg \cite{ff1,ff2} or
Diaconu--Goldfeld--Hoffstein \cite{dgh}. The first named author used the
$A_5$ quadratic series to study mean values of zeta functions of
biquadratic fields in \cite{chbq}.

In \cite{bfh-padic-whittaker} Bump, Friedberg, and Hoffstein computed
the Whittaker coefficients of an Eisenstein series on the double cover
of $GSp(2r).$ The Eisenstein series they studied was induced from a
(nonmetaplectic) cuspform on $GL(r).$ The authors found that the
Whittaker coefficients of these Eisenstein series involved quadratic
twists of the $L$-function of $f.$ In particular, these Whittaker
coefficients have Euler products.  According to the Eisenstein
conjecture stated in the introduction, when the cuspform $f$ is
replaced by a minimal parabolic Eisenstein series on $GL(r)$, the
Whittaker coefficients of the induced metaplectic Eisenstein series on
the double cover of $GSp(2r)$ are expected to coincide with the
quadratic multiple Dirichlet series associated to the root system
$B_r.$ Though a Weyl group multiple Dirichlet series does not
generally have an Euler product, it is easy to see from the relations
\eqref{eq:3},\eqref{eq:4},\eqref{eq:5} that for $n=2$ and root systems
of type $B$ (normalized so that the short roots have length one), the
function $H$ is actually multiplicative, not just twisted
multiplicative.  Therefore, modulo the validity of the Eisenstein
conjecture, our formula \ref{eqn:defofh} can be seen as analogue of
the Casselman--Shalika formula in this metaplectic context.  Bump,
Friedberg, and Hoffstein also give a Casselman--Shalika formula in
\cite{bfh-padic-whittaker}.  In work in preparation, Brubaker, Bump,
Chinta and Gunnells check that the formula of
\cite{bfh-padic-whittaker} agrees with the formula of this paper for
type $B$ and $n=2.$

\subsection{Cubic series} When $n=3$, the Kubota Dirichlet series
$\D(s,a)$ has the nice property that for $a\in \OO_S$ squarefree, its
residue at $s=4/3$ is, up to a constant and a power of $a$, the
conjugate Gauss sum $\overline{g(1,a)}$ (cf.~\cite{pat77a, pat77b}).
This was exploited by Brubaker and Bump \cite{bbFHL} to show that
residues of the $A_3,\ n=3$ series give two double Dirichlet series
involving $L$-functions associated to cubic characters studied by
Friedberg, Hoffstein, and Lieman \cite{fhl}. Using similar reasoning,
we expect that a triple residue of the $E_6, n=3$ series will give the
multiple Dirichlet series in three variables studied by Brubaker
\cite{brubaker-thesis}.  Labelling the outer nodes with the indices
1,2,3, the $E_6$ series is of heuristically of the form
\begin{equation}
  \label{eq:1}
  \sum_{a_1,\ldots, a_6}
  \frac{g(a_4, a_1)g(a_5, a_2)g(a_6, a_3)g(1, a_4)g(1, a_5)g(1, a_6)
    \overline{\prs{a_4}{a_5}}\,\overline{\prs{a_6}{a_5}}}
  {|a_1|^{s_1}\cdots |a_6|^{s_6}}.
\end{equation}
Taking residues in $s_1, s_2$ and $s_3$, the Gauss sums disappear and
we expect to be  left with a series of the form
\begin{equation}
  \label{eq:23}
  \sum_{a_4, a_5, a_6}\frac{
    \overline{\prs{a_4}{a_5}}\,\overline{\prs{a_6}{a_5}}}
  {|a_4|^{s_4'}|a_5|^{s_5'}|a_6|^{s_6'}}
\end{equation}
for some new variables $s_4', s_5', s_6'.$ The squarefree coefficients
coincide with those of a series used by Brubaker to prove an
asymptotic formula for second moment of cubic Dirichlet $L$-series.
Presumably, using the explicit description of the $\p$-part we have
given in Section \ref{sec:Waction} and using the techniques of
\cite{bbFHL} to compute the $\p$-parts of residues of the cubic
series, we can show that residues of the $E_6,\, n=3$ series coincide
with the series studied by Brubaker.  We have not checked this in
detail. 

\subsection{Higher $n$} Friedberg, Hoffstein, and Lieman \cite{fhl}
have constructed double Dirichlet series built out of $L$-functions of
order $n$ Hecke characters.  As mentioned above, Brubaker and Bump
show in \cite{bbFHL} that, when $n=3$, these series arise after taking
a single residue of the $A_3,\ n=3$ series.  They further conjecture
that for general $n$, the series of \cite{fhl} are $n-2$ fold residues
of the degree $n$ $A_n$ series.  The methods of \cite{bbFHL} do not
work for $n>3$ because of our incomplete understanding of the residues
of Kubota's Dirichlet series.  However, exploiting the fact that a
Weyl group multiple Dirichlet series over a function field can be
explicitly computed as a rational function, J. Mohler has verified the
conjecture of \cite{bbFHL} over the rational function field for $n\leq
9.$ It is likely that his techniques will yield a proof for all $n$
in the setting of a rational function field.

\subsection{The Weyl character formula}\label{ss:weights} The
construction of the rational function $h (\x ;\ell)$ in Theorem
\ref{thm:invRatlFunc} suggests that it should be thought of as a
deformation of the Weyl character formula.  Indeed, it is not hard to
prove
\[
\Delta (\x) = \sum_{w\in W} j (w,\x).
\]
With this result the definition of $h$ becomes 
\[
h (\x ; \ell) = \frac{\sum_{w\in W} j (w,\x) (1|_{\ell}w)
(\x)}{\sum_{w\in W} j (w, \x)}, 
\]
which is clearly analogous to the Weyl character formula.  

More precisely, let $\g$ be the complex semisimple Lie algebra
determined by $\Phi$, and let $\wt_{1}, \dotsc , \wt_{r}$ be the
fundamental weights of $\Phi$.  Write $h (\x ; \ell) = N (\x ; \ell)/
D (\x)$.  If $n=1$ then it turns out that $N (\x ; \ell )$ is actually
divisible by $D (\x)$, so that $h (\x ; \ell)$ is actually a
polynomial.  After some simple changes of variables and introduction
of $q$-powers, this polynomial can be identified with the character
$\chi_{\theta}$ of the representation $V_{\theta}$ of $\g$ of lowest
weight $-\theta$, where as before $\theta = \sum (l_{i} + 1)\wt_{i}$.

On the other hand, if $n>1$, then $N (\x ; \ell)$ is not divisible by
$D (\x)$ in general, and so $h$ is not a polynomial.  Nevertheless,
one still might view $h$ as a deformation of a character.  We plan to
explore this connection between characters, $N (\x ; \ell)$, and
$h(\x; \ell)$ in future work.

\bibliographystyle{alpha}
\bibliography{wmdsn}
\end{document}